\documentclass[12pt]{article}
\usepackage{amsmath,amssymb,amsthm}
\usepackage{fancyvrb}
\usepackage{bbold}

\usepackage[bottom,multiple]{footmisc}

\makeatletter
\let\@@pmod\pmod
\DeclareRobustCommand{\pmod}{\@ifstar\@pmods\@@pmod}
\def\@pmods#1{\mkern4mu({\operator@font mod}\mkern 6mu#1)}
\makeatother

\newtheorem{theorem}{Theorem}[section]
\newtheorem{lemma}{Lemma}[section]
\newtheorem{corollary}{Corollary}[section]

\begin{document}

\title{Simultaneous Cubic and Quadratic Diagonal Equations In 12 Prime Variables}

\author{Alan Talmage\footnote{abt5217@psu.edu \newline Department of Mathematics, Pennsylvania State University, State College, PA, USA \newline ORCID ID: 0000-0002-6025-8092}}

\maketitle

\begin{abstract}
The system of equations 
\[
		u_1p_1^2 + \ldots + u_sp_s^2 = 0
\]
\[
		v_1p_1^3 + \ldots + v_sp_s^3 = 0
\]
has prime solutions $(p_1, \ldots, p_s)$ for $s \geq 12$, assuming that the system has solutions modulo each prime $p$.  This is proved via the Hardy-Littlewood circle method, building on Wooley's work on the corresponding system over the integers and recent results on Vinogradov's mean value theorem.  Additionally, a set of sufficient conditions for local solvability is given: If both equations are solvable modulo 2, the quadratic equation is solvable modulo 3, and for each prime $p$ at least 7 of each of the $u_i$, $v_i$ are not zero modulo $p$, then the system has solutions modulo each prime $p$.
\end{abstract}

\section{Introduction}

Much work has been done in applying the Hardy-Littlewood circle method to find integral solutions to systems of simultaneous equations (see~\cite{bruderncook},~\cite{cook},~\cite{wooleySAEi}, and~\cite{wooleyDiagEq} for examples).  In particular, recent progress on Vinogradov's mean value theorem (see~\cite{bdg},~\cite{wooleyCubicCase}) has enabled progress on questions of this type.  Here we consider the question of solving systems of equations with prime variables, generalizing the Waring-Goldbach problem in the same way existing work on integral solutions of systems of equations generalizes Waring's problem.  Following Wooley~\cite{wooleyDiagEq}, we address here the simplest nontrivial case: one quadratic equation and one cubic equation.  We find that under suitable local conditions, 12 variables will suffice for us to establish an eventually positive asymptotic formula guaranteeing solutions to the system of equations.

Consider a pair of equations of the form
\begin{equation} \label{eq:system}
	\begin{split}
		u_1p_1^2 + \ldots + u_sp_s^2 = 0\\
		v_1p_1^3 + \ldots + v_sp_s^3 = 0
	\end{split}
\end{equation}
where $u_1, \ldots, u_s, v_1, \ldots, v_s$ are nonzero integer constants and $p_1, \ldots, p_s$ are variables restricted to prime values.  We seek to prove the following theorem:

\begin{theorem}\label{mainThm} If

\begin{enumerate}

\item the system~(\ref{eq:system}) has a nontrivial real solution,

\item $s \geq 12$, and

\item for every prime $p$, the corresponding local system
\begin{equation} \label{localSystem}
	\begin{split}
		u_1x_1^2 + \ldots + u_sx_s^2 = 0 \pmod p\\
		v_1x_1^3 + \ldots + v_sx_s^3 = 0 \pmod p
	\end{split}
\end{equation}
has a solution $(x_1, \ldots, x_s)$ with all $x_i \neq 0 \pmod p$,

\end{enumerate}

then the system has a solution $(p_1, \ldots, p_s)$ with all $p_i$ prime.  Moreover, if we let $R(P)$ be the number of solutions $(p_1, \ldots, p_s)$ with each $p_i \leq P$, each weighted by $(\log p_1)\ldots(\log p_s)$, then we have $R(P) \sim CP^{s-5}$ for some constant $C > 0$ uniformly over all choices of $u_1, \ldots, u_s, v_1, \ldots, v_s$.
\end{theorem}

In Section~\ref{sec:localConditions} we give a sufficient condition for (\ref{localSystem}) to be satisfied, giving us the explicit theorem

\begin{theorem}\label{mainThmv2} Consider the system

\begin{equation} \label{generalSystem}
	\begin{split}
		u_1p_1^2 + \ldots + u_sp_s^2 = U\\
		v_1p_1^3 + \ldots + v_sp_s^3 = V
	\end{split}
\end{equation}
where $u_1, \ldots, u_s, v_1, \ldots, v_s$, are nonzero integer constants and $U$, $V$ are integer constants.  If

\begin{enumerate}

\item the system has a nontrivial real solution,

\item $s \geq 12$,

\item the quadratic form $u_1p_1^2 + \ldots + u_sp_s^2$ is indefinite,

\item $\displaystyle\sum_{i=1}^s u_i = U \pmod 2$ and $\displaystyle\sum_{i=1}^s v_i = V \pmod 2,$

\item $\displaystyle\sum_{i=1}^s u_i = U \pmod 2$, and

\item for each prime $p \neq 2$, at least 7 of each of the $u_i$ and the $v_i$ are not zero modulo $p$,

\end{enumerate}

then the system has a solution $(p_1, \ldots, p_s)$ with all $p_i$ prime.  Moreover, if we let $R(P)$ be the number of solutions $(p_1, \ldots, p_s)$, each weighted by $(\log p_1)\ldots(\log p_s)$, then we have $R(P) \sim CP^{s-5}$ where $C > 0$ uniformly over all choices of $u_1, \ldots, u_s, v_1, \ldots, v_s$, $U$, and $V$.

\end{theorem}

We use the Hardy-Littlewood circle method to prove these results.  Section~\ref{sec:setup} performs the necessary setup for the application of the circle method: defining the relevant functions and the major arc/minor arc dissection.  Section~\ref{sec:lemmas} consists of a number of preliminary lemmas, which are referenced throughout.  Section~\ref{sec:minorArcs} proves a Hua-type bound necessary for the minor arcs.  Section~\ref{sec:Fbound} proves a Weyl-type bound on the minor arcs by means of Vaughan's identity.  Section~\ref{sec:majorArcs} is the circle method reduction to the singular series and singular integral.  Section~\ref{sec:singularSeriesConvergence} shows the convergence of the singular series and Section~\ref{sec:singularSeriesPositivity} shows that it is eventually positive, contingent on the local solvability of the system (\ref{generalSystem}).  Section~\ref{sec:localConditions} shows sufficient conditions for the solvability of the local system.  This depends on a computer check of local solvability for a finite number of primes.  Section~\ref{sec:computationalTechniques} discusses several techniques which can be employed to improve the efficiency of this computation.  Section~\ref{sec:conclusion} finishes the proof of Theorems~\ref{mainThm} and~\ref{mainThmv2}.  Appendix~\ref{codeAppendix} contains the source code used to run the computations laid out in Section~\ref{sec:computationalTechniques}.

\section{Notation and Definitions}\label{sec:setup}

As is standard in the literature, we use $e(\alpha)$ to denote $e^{2\pi i \alpha}$.  The letter $p$ is assumed to refer to a prime wherever it is used, and $\varepsilon$ means a sufficiently small positive real number.  The symbols $\Lambda$ and $\mu$ are the von Mangoldt and M\"obius functions, respectively.  Symbols in bold are tuples, with the corresponding symbol with a subscript denoting a component, i.e., $\mathbf{a} = (a_1, \ldots, a_k)$.  The letter $C$ is used to refer to a positive constant, with the value of $C$ being allowed to change from line to line.  We write $f(x) \ll g(x)$ for $f(x) = O(g(x))$, $f(x) \asymp g(x)$ if both $f(x) \ll g(x)$ and $g(x) \ll f(x)$ hold, and $f(x) \sim g(x)$ if $f(x)/g(x) \rightarrow 1$ as $x \rightarrow \infty$.  When we refer to a solution of the system under study, we mean an ordered $s$-tuple of prime numbers $(p_1, \ldots, p_s)$ satisfying~(\ref{generalSystem}).

Define the generating function 
\[
	f_i(\alpha, \beta) = \sum_{p \leq P} (\log p) e(\alpha u_i p^2 + \beta v_i p^3).
\]

Let $\mathcal{A}$ be the unit square $(\mathbb{R}/\mathbb{Z})^2$ and let
\begin{equation}\label{Rdef}
	R(P) = \int_{\mathcal{A}} \prod_{i=1}^s f_i(\alpha, \beta) d\alpha d\beta
\end{equation}
\[
	= \int_{\mathcal{A}} \sum_{p_1, \ldots, p_s \leq P} \prod_{i=1}^s \left( (\log p_i) e(\alpha u_i p_i^2 + \beta v_i p_i^3) \right) d\alpha d\beta
\]
\[
	= \sum_{\substack{\{p_1, \ldots, p_s\} \\ satisfies~(\ref{eq:system})}} \prod_{i=1}^s (\log p_i).
\]
Thus $R(P)>0$ if and only if there is a solution to the system~(\ref{eq:system}).

We divide $\mathcal{A}$ into major and minor arcs.  For any $T$ with $1 \leq T \leq P$. and for all $q < T$, $1 \leq a \leq q$, $1 \leq b \leq q$, $(a, b, q) = 1$, let a typical major arc $\mathfrak{M}(a,b,q;T)$ consist of all $(\alpha, \beta)$ such that

\[
	|\alpha-a/q| \leq \frac{T}{qP^2} \qquad \text{and} \qquad |\beta-b/q| \leq \frac{T}{qP^3}.
\]
Let the major arcs $\mathfrak{M}(T)$ be the union of all such $\mathfrak{M}(a,b,q)$, and let the minor arcs $\mathfrak{m}(T)$ be the remainder of $\mathcal{A}$.

We will use two distinct dissections in our argument: the primary dissection into $\mathfrak{M} = \mathfrak{M}(Q)$ and $\mathfrak{m} = \mathfrak{m}(Q)$ with $Q = (\log P)^A$, where $A$ is a positive constant whose value will be fixed later, and a secondary dissection $\mathfrak{M}(R)$, $\mathfrak{m}(R)$ with $R = P^{\frac{1}{2}+\varepsilon}$.

\section{Preliminary Lemmas}\label{sec:lemmas}

We begin by defining the necessary generating functions.  Let 
\[
	f(\pmb\alpha) = \sum_{P < p \leq 2P} e(\alpha_2p^2 + \alpha_3p^3),
\]
\[
	g(\pmb \alpha) = \sum_{P < n \leq 2P} e(\alpha_2n^2 + \alpha_3n^3),
\]

\begin{equation}\label{SDef}
		S(q,\pmb{a}) = \sum_{n=1}^q e\left(\frac{a_2n^2 + a_3n^3}{q}\right),
\end{equation}
\begin{equation}\label{WDef}
		W(q,\pmb{a}) = \sum_{\substack{n=1 \\ (n,q)=1}}^q e\left(\frac{a_2n^2 + a_3n^3}{q}\right),
\end{equation}
\[
	v(\pmb\theta) = \int_P^{2P} e(\theta_2x^2 + \theta_3x^3) dx,
\]
and for $\pmb\gamma \in \mathfrak{M}(R)$ let
\[
	V(\pmb\gamma) = \frac{1}{q} S(q,\pmb{a}) v\left(\gamma_2 - \frac{a_2}{q}, \gamma_3 - \frac{a_3}{q}\right).
\]

\begin{lemma}\label{g1012Bounds} We have the bounds
\[
	\int_\mathcal{A} |g(\pmb\alpha)|^{10} d\pmb\alpha \ll P^{\frac{31}{6}+\varepsilon}
\]
and
\[
	\int_\mathcal{A} |g(\pmb\alpha)|^{12} d\pmb\alpha \ll P^7.
\]
\end{lemma}
\begin{proof}
This is the relevant portion of Theorem 1.3 of~\cite{wooleyDiagEq}.
\end{proof}

\begin{lemma}\label{f1012Bounds} We have the bounds
\[
	\int_\mathcal{A} |f(\pmb\alpha)|^{10} d\pmb\alpha \ll P^{\frac{31}{6}+\varepsilon} d\pmb\alpha
\]
and
\[
	\int_\mathcal{A} |f(\pmb\alpha)|^{12} d\pmb\alpha \ll P^7 d\pmb\alpha.
\]
\end{lemma}
\begin{proof}
For any positive integer $k$,
\[
	\int_\mathcal{A} |g(\pmb\alpha)|^{2k} d\pmb\alpha
\]
is the number of positive integer solutions to the system
\[
	p_1^2 + \ldots + p_k^2 = p_{k+1}^2 + \ldots + p_{2k}^2
\]
\[
	p_1^3 + \ldots + p_k^3 = p_{k+1}^3 + \ldots + p_{2k}^3
\]
and 
\[
	\int_\mathcal{A} |f(\pmb\alpha)|^{2k} d\pmb\alpha
\]
is the number of prime solutions to the same system, so this lemma follows from Lemma~\ref{g1012Bounds}.
\end{proof}

\begin{lemma}\label{pointwisegBound}
\[
	\sup_{\pmb\alpha\in\mathfrak{m}(R)} |g(\pmb\alpha)| \ll P^{\frac{5}{6} - \frac{\delta}{3} + \varepsilon}.
\]
\end{lemma}
\begin{proof}
This follows from Lemma 5.2 of~\cite{wooleyDiagEq}.
\end{proof}

\begin{lemma}\label{vBound}
\[
	v(\pmb\theta) \ll \frac{P}{(1 + P^3|\theta_3|)^{1/2}}.
\]
\end{lemma}
\begin{proof}
If $|\theta_3| \leq P^{-3}$, the result is immediate.  Thus we assume $|\theta_3| > P^{-3}$.  Let $K = (|\theta_3| P)^{\frac{1}{2}}$ and let $r(x) = \theta_2x^2+\theta_3x^3$.  Then $r'(x) = 2\theta_2x + 3\theta_3x^2$ has at most one zero in $[P,2P]$.  Thus we can divide $[P,2P]$ into subsets $I_1$ and $I_2$ such that $|r'(x)| \geq K$ on $I_1$, where $I_1$ is the union of at most three intervals $J$ such that $r'(x)$ is monotonic on each, and $|r'(x)| \leq K$ on $I_2$, where $I_2$ is the union of at most two intervals.

First we consider $I_1$:
\[
	\int_{I_1} e(r(x)) dx = \int_{I_1} \frac{1}{2\pi i r'(x)} \frac{d}{dx}e(r(x)) dx,
\]
so, upon integrating by parts,
\[
	\int_{I_1} e(r(x)) dx = \frac{e(r(x))}{2\pi i r'(x)} \bigg|_{I_1} + \int_{I_1} \frac{r''(x)}{2\pi i r'(x)^2} e(r(x)) dx.
\]
The integral on the right is bounded by
\[
	\int_{I_1} \frac{|r''(x)|}{2\pi r'(x)^2} dx = \left|\int_{I_1} \frac{r''(x)}{2\pi r'(x)^2} dx \right| = \left|\frac{-1}{2\pi r'(x)} \bigg|_{I_1} \right| \ll \frac{1}{K},
\]
since $r'(x)$ is monotonic on each interval in $I_1$.  Thus
\begin{equation}{\label{I1Bound}
	\int_{I_1} e(r(x)) dx \ll \frac{e(r(x))}{2\pi i r'(x)} \bigg|_{I_1} + \frac{1}{K} \ll \frac{1}{K} \ll \frac{P}{(1+|\theta_3|P^3)^{1/2}}}.
\end{equation}

Next we consider $I_2$.  Given an interval in $I_2$, let $x_0$ be one of its endpoints.  
Then for any $x$ in $I_2$,
\[
	|x-x_0||2\theta_2 + 3\theta_3(x+x_0)| = |r'(x) - r'(x_0)| \leq 2K.
\]
Moreover,
\begin{equation}\label{phiBoundedByLambda}
	|2\theta_2+3\theta_3x_0| = \frac{|r'(x_0)|}{x_0} \leq \frac{K}{x_0}.
\end{equation}
Applying the triangle identity to (\ref{phiBoundedByLambda}) yields
\begin{equation}\label{phiDiffUpperBound}
	|3\theta_3x| - \frac{K}{x_0} \leq |2\theta_2+3\theta_3(x+x_0)|.
\end{equation}
Also,
\begin{equation}\label{phiDiffLowerBound}
	|3\theta_3x| - \frac{K}{x_0} \geq 3|\theta_3|P - \frac{K}{P} \geq 2|\theta_3|P.
\end{equation}
Combining (\ref{phiBoundedByLambda}), (\ref{phiDiffUpperBound}), and (\ref{phiDiffLowerBound}) yields
\[
	|x-x_0| \leq \frac{2K}{2|\theta_3|P} = \frac{P}{(|\theta_3|P^3)^{1/2}}.
\]
Thus
\begin{equation}\label{I2Bound}
	\begin{split}
		\int_{I_2} e(r(x)) dx 	& \ll |e(r(x))| \left(\operatorname{meas} (I_2)\right) \\ 
									& \ll 2\max_{x\in I_2} |x-x_0| \\
									& \ll \frac{P}{(1 + |\theta_3|P^3)^{1/2}}.
	\end{split}
\end{equation}

Combining (\ref{I1Bound}) and (\ref{I2Bound}) now gives the desired result.

\end{proof}

\begin{lemma} Let $t = 12-\delta$.  Then
\[
	\int_\mathcal{A} |f(\pmb\alpha)|^{t-1} d\pmb\alpha \ll P^{t-6+\frac{1+\delta}{12}+\varepsilon}.
\]
\end{lemma}
\begin{proof}
By H\"older's inequality
\[
	\int_\mathcal{A} |f(\pmb\alpha)|^{t-1} d\pmb\alpha \leq \left(\int_\mathcal{A} |f(\pmb\alpha)|^{12} d\pmb\alpha\right)^{\frac{t-11}{2}} \left(\int_\mathcal{A} |f(\pmb\alpha)|^{10} d\pmb\alpha\right)^{\frac{13-t}{2}}.
\]
Applying Lemma~\ref{f1012Bounds} gives
\[
	\int_\mathcal{A} |f(\pmb\alpha)|^{t-1} d\pmb\alpha \ll P^{\frac{7t-77}{2} + \frac{403-31t}{12} + \varepsilon} = P^{t-6+\frac{1+\delta}{12}+\varepsilon}.
\]

\end{proof}

\begin{lemma}\label{majorArcgBound}
Let $R = P^{\frac{1}{2}+\delta}$ and let $\pmb\gamma \in \mathfrak{M}(R)$.  Then
\[
	g(\pmb\gamma) = V(\pmb\gamma) + O\left(P^{\frac{5}{6} - \frac{\delta}{3}}\right).
\]
\end{lemma}
This follows from Theorem 7.2 of \cite{vaughanHLM}.

\begin{lemma}\label{kappaDef} Let $\kappa(q)$ be the multiplicative function defined by
\[
	\kappa(p^j) = \begin{cases}
	C p^{-1/2} & j=1, \\
	C p^{-5/8} & j=2, \\
	C p^{-j/4} & j>2.
	\end{cases}
\]
Then there is a positive constant $C$ such that
\[
	\max_{\substack{\mathbf{a} \\ (q,a_2,a_3)=1}}\frac{|S(q,\mathbf{a})|}{q} \leq \kappa(q).
\]
\end{lemma}
\begin{proof}
The case $j=1$ follows from Theorem 2E of~\cite{schmidt}.  The cases with $j > 1$ follow from Theorem 7.1 of~\cite{vaughanHLM}.
\end{proof}

Let
\begin{equation}\label{sDef}
	s_k(\mathbf{m}) = m_1^k + m_2^k + m_3^k - m_4^k - m_5^k - m_6^k.
\end{equation}

\begin{lemma}\label{rogLemma}Let $Q>0$ and let $M(Q)$ be the number of solutions of the system
\[
	s_2(\mathbf{m}) = 0
\]
\[
	s_1(\mathbf{m}) = 0
\]
with all $m_j \leq Q$.  Then there is a positive constant $C$ such that
\[
	M(Q) \sim C Q^3 \log Q.
\]
\end{lemma}
This is a result of Rogovskaya~\cite{rogovskaya}.

\begin{lemma}\label{WqaBound} If $(q, a_2, a_3)=1$, then
\[
	W(q, \mathbf{a}) \ll q^{\frac{1}{2}+\varepsilon}.
\]
In addition, if $(p,a_2,a_3)=1$, then
\[
	W(p, \mathbf{a}) \ll p^{\frac{1}{2}}.
\]

\end{lemma}

\begin{proof}

The case where $q = p$ follows from Theorem 2E of~\cite{schmidt}.  The case for general $q$ follows from Lemma 8.5 of~\cite{hua}.

\end{proof}

\section{Minor Arc Bounds}\label{sec:minorArcs}

The primary purpose of this section is to prove the following theorem, which, together with the result of the next section, will provide the necessary minor arc bounds for our circle method approach.

\begin{theorem}\label{huaBound}
Let $\delta < 1$ be a small positive number, and let $t = 12 - \delta$.  Then
\[
	\int_\mathcal{A} |f(\pmb\alpha)|^t d\pmb\alpha \ll P^{t-5}(\log P).
\]
\end{theorem}

Let $t = 12-\delta$ for some small $\delta > 0$, and let 
\[
	I_t(P) = \int_\mathcal{A} |f(\pmb\alpha)|^t d\pmb\alpha.
\]

\begin{lemma}\label{theTrick}
\[
	I_t(P)^2 \ll P^{2t-10} + P\mathop{\int_\mathcal{A} \int_\mathcal{A}}_{\pmb\alpha-\pmb\beta\in\mathfrak{M}(R)} |V(\pmb\alpha - \pmb\beta)||f(\pmb\alpha)|^{t-1}|f(\pmb\beta)|^{t-1} d\pmb\alpha d\pmb\beta.
\]
\end{lemma}
\begin{proof}
\[
	I_t(P) = \int_\mathcal{A} f(\pmb\alpha)f(-\pmb\alpha)|f(\pmb\alpha)|^{t-2} d\pmb\alpha
\]
\begin{equation}\label{trickStart}
	= \sum_{P < p \leq 2P}\int_\mathcal{A} e(\alpha_2p^2 + \alpha_3p^3) f(-\pmb\alpha)|f(\pmb\alpha)|^{t-2} d\pmb\alpha.
\end{equation}

Applying the Cauchy-Schwarz identity to (\ref{trickStart}) yields
\[
	I_t(P)^2 \ll P\sum_{P < n \leq 2P} \left| \int_\mathcal{A} e(\alpha_2n^2 + \alpha_3n^3) f(-\pmb\alpha)|f(\pmb\alpha)|^{t-2} d\pmb\alpha \right|^2
\]
\[
	= P \int_\mathcal{A}\int_\mathcal{A} g(\pmb\alpha-\pmb\beta) f(-\pmb\alpha)|f(\pmb\alpha)|^{t-1} f(\pmb\beta)|f(\pmb\beta)|^{t-1} d\pmb\alpha d\pmb\beta
\]
\begin{equation}\label{trickReady}
	\leq P \int_\mathcal{A}\int_\mathcal{A} |g(\pmb\alpha-\pmb\beta)| |f(\pmb\alpha)|^{t-1} |f(\pmb\beta)|^{t-1} d\pmb\alpha d\pmb\beta.
\end{equation}

By Lemma~\ref{pointwisegBound} and recalling that $t = 12-\delta$, we can bound the minor arc portion of (\ref{trickReady}):
\[
	 P \mathop{\int_\mathcal{A}\int_\mathcal{A}}_{\pmb\alpha-\pmb\beta \in \mathfrak{m}(R)} |g(\pmb\alpha-\pmb\beta)| |f(\pmb\alpha)|^{t-1} |f(\pmb\beta)|^{t-1} d\pmb\alpha d\pmb\beta
\]
\[
	\ll P^{\frac{11}{6} - \frac{\delta}{3}+\varepsilon} \left(\int_\mathcal{A} |f(\pmb\alpha)|^{t-1}d\pmb\alpha\right)^2
\]
\[
	\ll P^{2t-10+\frac{13-t}{6}-\frac{1}{6}-\frac{\delta}{3}+2\varepsilon}
\]
\begin{equation}\label{minorArcItP2}
	\ll P^{2t-10-\frac{\delta}{6}+2\varepsilon} \ll P^{2t-10}.
\end{equation}

We now apply Lemma~\ref{majorArcgBound} to the major arc portion of (\ref{trickReady}).
\[
	P \mathop{\int_\mathcal{A}\int_\mathcal{A}}_{\pmb\alpha-\pmb\beta \in \mathfrak{M}(R)} |g(\pmb\alpha-\pmb\beta)| |f(\pmb\alpha)|^{t-1} |f(\pmb\beta)|^{t-1} d\pmb\alpha d\pmb\beta
\]
\[
	= P \mathop{\int_\mathcal{A}\int_\mathcal{A}}_{\pmb\alpha-\pmb\beta \in \mathfrak{M}(R)} |V(\pmb\alpha-\pmb\beta)| |f(\pmb\alpha)|^{t-1} |f(\pmb\beta)|^{t-1} d\pmb\alpha d\pmb\beta
\]
\[
+ O\left(P^{\frac{11}{6}-\frac{\delta}{3}}  \left(\int_\mathcal{A} |f(\pmb\alpha)|^{t-1}d\pmb\alpha\right)^2 \right)
\]
\begin{equation}\label{majorArcItP2}
	= P \mathop{\int_\mathcal{A}\int_\mathcal{A}}_{\pmb\alpha-\pmb\beta \in \mathfrak{M}(R)} |V(\pmb\alpha-\pmb\beta)| |f(\pmb\alpha)|^{t-1} |f(\pmb\beta)|^{t-1} d\pmb\alpha d\pmb\beta + O(P^{2t-10}).
\end{equation}

Combining (\ref{trickReady}), (\ref{minorArcItP2}), and (\ref{majorArcItP2}) yields the lemma.

\end{proof}

Let $\pmb\gamma = \pmb\alpha - \pmb\beta$,
\begin{equation}\label{lambdaDef}
	\lambda = \frac{t-6}{2} = 3 - \frac{\delta}{2}
\end{equation}
(note that $\lambda > 2$), and
\begin{equation}\label{JDef}
	J(\pmb\beta) = \int_{\mathfrak{M}(R)} |V(\pmb\gamma)|^\lambda |f(\pmb\beta + \pmb\gamma)|^6 d\pmb\gamma.
\end{equation}

\begin{lemma}\label{IBoundedByJ}
\[
	I_t(P) \ll P^{t-5} + P^\lambda \sup_{\pmb\beta \in \mathcal{A}} J(\pmb\beta).
\]
\end{lemma}
\begin{proof}

We begin by noting that
\[
	|V(\pmb\alpha-\pmb\beta)| |f(\pmb\alpha)|^{t-1} |f(\pmb\beta)|^{t-1}
\]
can be rewritten as
\begin{align}\label{VIntegrandBreakdown}
\begin{split}
	& \left(|V(\pmb\alpha-\pmb\beta)|^\lambda |f(\pmb\alpha)|^6 |f(\pmb\beta)|^t\right)^{\frac{1}{2\lambda}} \\
	\times & \left(|V(\pmb\alpha-\pmb\beta)|^\lambda |f(\pmb\beta)|^6 |f(\pmb\alpha)|^t\right)^{\frac{1}{2\lambda}} \\
	\times & \left(|f(\pmb\alpha)f(\pmb\beta)|^t\right)^{1-\frac{1}{\lambda}}.
\end{split}
\end{align}

Let 
\[
	I_t^*(P) = \mathop{\int_\mathcal{A} \int_\mathcal{A}}_{\pmb\alpha-\pmb\beta\in\mathfrak{M}(R)} |V(\pmb\alpha - \pmb\beta)||f(\pmb\alpha)|^{t-1}|f(\pmb\beta)|^{t-1} d\pmb\alpha d\pmb\beta
\]
be the integral on the right in Lemma~\ref{theTrick}.  Using (\ref{VIntegrandBreakdown}) to apply H\"older's inequality to $I_t^*(P)$, we obtain
\begin{align}
\begin{split}
	I_t^*(P) \ll & \left( \mathop{\int_\mathcal{A}\int_\mathcal{A}}_{\pmb\alpha-\pmb\beta\in\mathfrak{M}(R)} |V(\pmb\alpha-\pmb\beta)|^\lambda |f(\pmb\alpha)|^6 |f(\pmb\beta)|^t d\pmb\alpha\pmb\beta \right)^\frac{1}{\lambda} \\
	& \times \left( \mathop{\int_\mathcal{A}\int_\mathcal{A}}_{\pmb\alpha-\pmb\beta\in\mathfrak{M}(R)} |f(\pmb\alpha)f(\pmb\beta)|^t d\pmb\alpha\pmb\beta \right)^{1-\frac{1}{\lambda}}
\end{split}
\end{align}

\begin{equation}\label{ItBoundedByIt}
	\leq I_t(P)^{2-\frac{1}{\lambda}} \left( \sup_{\pmb\beta\in\mathcal{A}} \int_{\mathfrak{M}(R)} |V(\pmb\gamma)|^\lambda |f(\pmb\beta + \pmb\gamma)|^6 d\pmb\gamma \right)^\frac{1}{\lambda}.
\end{equation}

Applying (\ref{ItBoundedByIt}) to Lemma~\ref{theTrick}, we have
\[
	I_t(P)^2 \ll P^{2t-10} + PI_t(P)^{2-\frac{1}{\lambda}}\left( \sup_{\pmb\beta\in\mathcal{A}} J(\pmb\beta) \right)^{\frac{1}{\lambda}}.
\]
Thus either $I_t(P) \ll P^{t-5}$ or
\[
	I_t(P) \ll P^\lambda \sup_{\pmb\beta\in\mathcal{A}} J(\pmb\beta),
\]
which implies the desired result.

\end{proof}

\begin{lemma}\label{JBound} Let $N(q)$ be the number of solutions of the system
\[
	\begin{cases} s_3(\mathbf p) & \equiv 0 \pmod q, \\ s_2(\mathbf p) & =0, \\ P < p_j \leq 2P. \end{cases}
\]
Then

\[
	J(\pmb\beta) \ll P^{\lambda-3}\sum_{q \leq R} \kappa(q)^\lambda q N(q).
\]
\end{lemma}
\begin{proof}
By (\ref{JDef}) and the definition of $\mathfrak{M}(R)$,
\[
	J(\pmb\beta) = \sum_{q \leq R} \mathop{\sum_{a_2=1}^q \sum_{a_3=1}^q}_{(q,a_2,a_3)=1} \frac{|S(q,\mathbf{a})|^\lambda}{q^\lambda} \int_{-\frac{R}{qP^2}}^{\frac{R}{qP^2}} \int_{-\frac{R}{qP^3}}^{\frac{R}{qP^3}} |v(\pmb\theta)|^\lambda \left|f\left(\pmb\beta + \frac{\mathbf{a}}{q} + \pmb\theta \right)\right|^6 d\pmb\theta.
\]
By Lemma~\ref{vBound}, Lemma~\ref{kappaDef}, and the fact that for a given $q$, the intervals $\left[\frac{a_2}{q} - \frac{R}{qP^2}, \frac{a_2}{q} + \frac{R}{qP^2}\right]$ are disjoint for distinct $a_2$, 
\begin{equation}\label{JBreakdown}
	J(\pmb\beta) \leq \sum_{q \leq R} \int_{-\frac{R}{qP^3}}^{\frac{R}{qP^3}} \frac{\kappa(q)^\lambda P^\lambda}{(1 + P^3|\theta_3|)^{\lambda/2}} \sum_{a_3=1}^q \int_0^1 \left| f\left(\beta_2 + \phi, \beta_3 + \frac{a_3}{q} + \theta_3\right)\right|^6 d\phi d\theta_3.
\end{equation}

We now examine the inner sum and integral.
\[
	\sum_{a_3=1}^q \int_0^1 \left| f\left(\beta_2 + \phi, \beta_3 + \frac{a_3}{q} + \theta_3\right)\right|^6 d\phi
\]
\[
	= \sum_{a_3=1}^q \int_0^1 \sum_{\substack{\mathbf p \\ P < p_j \leq 2P}} e\left( (\beta_2+\phi)s_2(\mathbf{p}) + \left(\beta_3 + \frac{a_3}{q} + \theta_3\right)s_3(\mathbf{p})\right) d\phi
\]
\[
	= \sum_{\substack{\mathbf p \\ P < p_j \leq 2P}} e(\beta_2s_2(\mathbf{p}) + (\beta_3+\theta_3)s_3(\mathbf{p})) \sum_{a_3=1}^q e\left(\frac{a_3}{q}s_3(\mathbf{p})\right) \int_0^1 e(\phi s_2(\mathbf{p}))d\phi.
\]

Now
\[
	\sum_{a_3=1}^q e\left(\frac{a_3}{q}s_3(\mathbf{p})\right) 
	= \begin{cases}
	0 & s_3(\mathbf{p}) \not\equiv 0 \pmod q, \\
	q & s_3(\mathbf{p}) \equiv 0 \pmod q
	\end{cases}
\]
and
\[
	\int_0^1 e(\phi s_2(\mathbf{p}))d\phi
	= \begin{cases}
	0 & s_2(\mathbf{p}) \neq 0, \\
	1 & s_2(\mathbf{p}) = 0,
	\end{cases}
\]
so
\begin{equation}\label{innerSumBoundedByqN}
	\sum_{a_3=1}^q \int_0^1 \left| f\left(\beta_2 + \phi, \beta_3 + \frac{a_3}{q} + \theta_3\right)\right|^6 d\phi \ll qN(q).
\end{equation}

Substituting (\ref{innerSumBoundedByqN}) into (\ref{JBreakdown}) yields
\[
	J(\pmb\beta) \ll \sum_{q \leq R} \kappa(q)^\lambda qN(q) \int_{-\frac{R}{qP^3}}^{\frac{R}{qP^3}} \frac{P^\lambda}{(1 + P^3|\theta_3|)^{\lambda/2}} d\theta_3.
\]
Since $\lambda > 2$, this becomes
\[
	J(\pmb\beta) \ll P^{\lambda-3}\sum_{q \leq R} \kappa(q)^\lambda qN(q).
\]

\end{proof}

\begin{lemma}\label{NToN1}
Let $N_1(q)$ be the number of solutions to the system
\[
	\begin{cases}
	s_3(\mathbf p) & \equiv 0 \pmod q, \\
	s_2(\mathbf p) & = 0, \\
	P < p_j \leq 2P & p_j \nmid q.
	\end{cases}
\]
Then
\[
	qN(q) \ll q(\log q)^6 + qN_1(q).
\]
\end{lemma}
\begin{proof}
First, note that
\[
	qN(q) = \sum_{a_3=1}^q \int_0^1 \left|f\left(x, \frac{a_3}{q}\right)\right|^6 dx.
\]
Let
\[
	f_|(\pmb\alpha) = \sum_{\substack{P < p < 2P \\ p|q}} e(\alpha_2p^2 + \alpha_3p^3)
\]
and
\[
	f_\nmid(\pmb\alpha)\sum_{\substack{P < p < 2P \\ p\nmid q}} e(\alpha_2p^2 + \alpha_3p^3).
\]
Thus
\[
	f(\pmb\alpha) = f_|(\pmb\alpha) + f_\nmid(\pmb\alpha).
\]
Since $|f_|(\pmb\alpha)| \ll \log q$,
\[
	|f(\pmb\alpha)|^6 \ll (\log q)^6 + |f_\nmid(\pmb\alpha)|^6.
\]
But $|f_\nmid(\pmb\alpha)|^6 = N_1(q)$, so
\[
	qN(q) \ll q(\log q)^6 + qN_1(q).
\]
\end{proof}

\begin{lemma}\label{N1ToN2} Let $N_2(q)$ be the number of solutions of the system
\begin{equation}\label{N2Def}
	\begin{cases}
	s_3(\mathbf r) & \equiv 0 \pmod q, \\
	s_2(\mathbf r) & \equiv 0 \pmod q, \\
	s_2(q\mathbf m + \mathbf r) & =0, \\
	1 \leq r_j \leq q & (q,r_j)=1, \\
	\frac{P-r_j}{q} < m_j \leq \frac{2P - r_j}{q}.
	\end{cases}
\end{equation}
Then
\[
	N_1(q) \leq N_2(q).
\]
\end{lemma}
\begin{proof}
We classify the solutions $\mathbf{p}$ counted by $N_1(q)$ according to the residue class $r_j$ of each $p_j$ modulo $q$, and let $m_j = \frac{p_j-r_j}{q}$.  
Thus
\[
	0 = s_2(q\mathbf{m} + \mathbf{r}) \equiv s_2(\mathbf{r}) \pmod q,
\]
so $N_1(q) \leq N_2(q)$.
\end{proof}

\begin{lemma}\label{N2ToN3} Let $N_3(q)$ be the number of solutions of the system
\[
	\begin{cases}
	s_3(\mathbf r) & \equiv 0 \pmod q, \\
	s_2(\mathbf r) & \equiv 0 \pmod q, \\
	1 \leq r_j \leq q & (q,r_j)=1.
	\end{cases}
\]
Then
\[
	N_2(q) \ll N_3(q)P^4q^{-5}(\log P)\left(\frac{q^2}{P} + 1\right).
\]
\end{lemma}

\begin{proof}

Let $\mathbf{r}$, $\mathbf{m}$ be a solution counted in $N_2(q)$, i.e., let $\mathbf{r}$, $\mathbf{m}$ satisfy (\ref{N2Def}).  Expanding the third equation of (\ref{N2Def}) gives
\[
	q^2s_2(\mathbf{m}) + 2q(r_1m_1 + r_2m_2 + r_3m_3 - r_4m_4 - r_5m_5 - r_6m_6) + s_2(\mathbf{r}) = 0.
\]
Since $s_2(\mathbf{r}) \equiv 0 \pmod q$ by the second equation of (\ref{N2Def}), this can be rewritten as
\[
	qs_2(\mathbf{m}) + 2(r_1m_1 + r_2m_2 + r_3m_3 - r_4m_4 - r_5m_5 - r_6m_6) + \frac{s_2(\mathbf{r})}{q} = 0
\]
with each term remaining integer-valued.  For a fixed $\mathbf{r}$, define
\begin{equation}\label{HjDef}
	H_j(\alpha) = \sum_{\frac{P-r_j}{q} < m \leq \frac{2P - r_j}{q}} e\left(\alpha(qm^2 + 2r_jm)\right).
\end{equation}
Thus the number of $\mathbf{m}$ satisfying (\ref{N2Def}) for a given $\mathbf{r}$ is
\[
	\int_0^1 H_1(\alpha)H_2(\alpha)H_3(\alpha)H_4(-\alpha)H_5(-\alpha)H_6(-\alpha) e\left(\frac{s_2(\mathbf{r})}{q} \alpha\right) d\alpha.
\]
By H\"older's inequality this is
\[
	\leq \prod_{j=1}^6 \left(\int_0^1 |H_j(\alpha)|^6 d\alpha \right)^{\frac{1}{6}}.
\]
The integral
\[
	\int_0^1 |H_j(\alpha)|^6 d\alpha
\]
counts the number of solutions of
\begin{equation}\label{rogSystem}
	qs_2(\mathbf{m}) + 2r_j s_1(\mathbf{m}) = 0.
\end{equation}
Let $s_2(\mathbf{m}) = u$ and $s_1(\mathbf{m}) = v$.  Then (\ref{rogSystem}) becomes $qu + 2r_j v = 0$.  For any solution, we have $|v| \leq \frac{6P}{q}$, and since $(q,r_j)=1$, $v = \frac{v'q}{(q,2)}$.  Thus the number of choices for $v'$ is $\leq 1 + 24P/q^2$, and $u$ is determined by $v'$.

Let
\[
	h(\pmb\alpha) = \sum_{\frac{P-r_j}{q} < m \leq \frac{2P - r_j}{q}} e(\alpha_1 m + \alpha_2 m^2).
\]
For fixed pair $u$, $v$, the number of choices of $\mathbf{m}$ is
\[
	\int_\mathcal{A} |h(\pmb\alpha)|^6 e(-\alpha_1v - \alpha_2u) d\pmb\alpha
\]
\[
	\leq \int_\mathcal{A} |h(\pmb\alpha)|^6 d\pmb\alpha.
\]
But this is the number of solutions of the system
\begin{equation}\nonumber
	\begin{split}
		& s_2(\mathbf{m}) = 0 \\
		& s_1(\mathbf{m}) = 0,
	\end{split}
\end{equation}
so by Lemma~\ref{rogLemma},
\[
	\int_\mathcal{A} |h(\pmb\alpha)|^6 d\pmb\alpha \ll \left(\frac{P}{q}\right)^3 \log P.
\]
So, given $\mathbf{r}$ satisfying the first two equations of (\ref{N2Def}) and $(q,r_j)=1$, the number of solutions to the third equation of (\ref{N2Def}) is
\[
	\ll \left(1 + \frac{P}{q^2}\right) \frac{P^3}{q^3} \log P = P^4q^{-5}\left(1 + \frac{q^2}{P}\right)\log P.
\]
Thus
\[
	N_2(q) \ll N_3(q)P^4q^{-5}(\log P)\left(1 + \frac{q^2}{P}\right).
\]

\end{proof}

\begin{lemma}\label{N3Bound}
Let $N_3(q)$ be as defined in Lemma~\ref{N2ToN3} above.  Then there exists a positive constant $C$ such that
\[
	N_3(q) \ll q^4 \prod_{p | q}\left(1 + \frac{C}{p}\right).
\]
\end{lemma}
\begin{proof}

We begin by observing that $N_3(q)$ is a multiplicative function, and that by orthogonality,
\[
	N_3(p^k) = p^{-2k} \sum_{b_2=1}^{p^k} \sum_{b_3=1}^{p^k} |W(p^k, b_2, b_3)|^6.
\]
Sorting the terms of this sum by value of $(p^k, b_2, b_3) = p^{k-j}$, where $0 \leq j \leq k$, gives
\[
	N_3(p^k) = p^{-2k} \sum_{j=0}^k \mathop{\sum_{a_2=1}^{p^j} \sum_{a_3=1}^{p^j}}_{(p^j,a_2,a_3)=1} |W(p^k, p^{k-j}a_2, p^{k-j}a_3)|^6.
\]

If $j=0$, then
\[
	W(p^k, p^{k-j}a_2, p^{k-j}a_3) = \phi(p^k) = p^k(1-1/p)
\]
and if $j > 0$, then
\[
	W(p^k, p^{k-j}a_2, p^{k-j}a_3) = p^{k-j}W(p^j,a_2,a_3).
\]
Thus
\[
	N_3(p^k) = p^{4k}(1-1/p)^6 + p^{4k} \sum_{j=1}^k \mathop{\sum_{a_2=1}^{p^j} \sum_{a_3=1}^{p^j}}_{(p^j,a_2,a_3)=1} p^{-6j} |W(p^j,a_2,a_3)|^6.
\]

By Lemma~\ref{WqaBound}, 
\[
	\sum_{\substack{\mathbf{a} \\ (p,a_2,a_3)=1}} p^{-6} |W(p,a_2,a_3)|^6 \ll p^{-1},
\]
and for $j \geq 2$,
\[
	\sum_{\substack{\mathbf{a} \\ (p^j,a_2,a_3)=1}} p^{-6j} |W(p^j,a_2,a_3)|^6 \ll p^{-4j+6j/2+j\varepsilon} \ll p^{-j + j\varepsilon}.
\]
Thus
\[
	N_3(p^k) \leq p^{4k}\left(1 + \frac{C}{p}\right)
\]
and the lemma follows by multiplicativity.

\end{proof}

\textit{Proof of Theorem~\ref{huaBound}: }
By Lemma~\ref{IBoundedByJ},
\[
	I_t(P) \ll P^{t-5} + P^\lambda \sup_{\pmb\beta\in\mathcal{A}} J(\pmb\beta).
\]
Bounding $J(\pmb\beta)$ with Lemma~\ref{JBound} yields:
\begin{equation}\label{ItToN}
	I_t(P) \ll P^{t-5} + P^{2\lambda-3}\sum_{q \leq R}\kappa(q)^\lambda q N(q).
\end{equation}
Lemmas~\ref{NToN1},~\ref{N1ToN2}, and~\ref{N2ToN3} succesively bound $N(q)$ in terms of $N_1(q)$, then $N_2(q)$, then $N_3(q)$, and Lemma~\ref{N3Bound} bounds $N_3(q)$.  Collecting these bounds and applying them to (\ref{ItToN}) gives
\[
	I_t(P) \ll P^{t-5} + P^{2\lambda + 1}(\log P)\sum_{q\leq R}\kappa(q)^\lambda \left( P^{-4}q(\log q)^6 + \left(\frac{q^2}{P}+1\right) \prod_{p|q}\left(1 + \frac{C}{p}\right)\right).
\]

Since $q \leq R = P^{\frac{1}{2}+\delta}$,
\[
	P^{-4}q(\log q)^6 \ll P^{-3} \ll 1,
\]
and
\[
	\frac{q^2}{P} \leq q^{\frac{4\delta}{1+2\delta}},
\]
so we have
\begin{equation}\label{ItTokappaSum}
I_t(P) \ll P^{t-5} + P^{2\lambda + 1}(\log P)\sum_{q\leq R}\kappa(q)^\lambda q^{\frac{4\delta}{1+2\delta}} \prod_{p|q}\left(1 + \frac{C}{p}\right).
\end{equation}

We now desire a bound on
\[
	\sum_{q\leq R}\kappa(q)^\lambda q^{\frac{4\delta}{1+2\delta}} \prod_{p|q}\left(1 + \frac{C}{p}\right).
\]
Since $\kappa$ is multiplicative, it suffices to bound
\[
	\prod_{p\leq R}\left(1 + \sum_{j=1}^{\infty}\kappa(p^j)^\lambda p^{j\frac{4\delta}{1+2\delta}}\right).
\]
We have
\[
	\sum_{j=1}^{\infty}\kappa(p^j)^\lambda p^{j\frac{4\delta}{1+2\delta}} \ll p^{-5/4} + p^{-3/2} + \sum_{j=3}^{\infty} p^{-\frac{2}{3}j}
\]
\[
	\ll p^{-5/4}.
\]
Thus
\[
	\prod_{p\leq R}\left(1 + \left(1+\frac{C}{p}\right)\sum_{j=1}^\infty\kappa(p^j)^\lambda p^{j\frac{4\delta}{1+2\delta}}\right) \ll \prod_{p\leq R}(1 + Cp^{-5/4}) \ll 1,
\]
which implies that
\begin{equation}\label{kappaSumBound}
	\sum_{q\leq R}\kappa(q)^\lambda q^{\frac{4\delta}{1+2\delta}} \prod_{p|q}\left(1 + \frac{C}{p}\right) \ll 1.
\end{equation}

Applying (\ref{kappaSumBound}) to (\ref{ItTokappaSum}) yields
\[
	I_t(P) \ll P^{t-5} + P^{2\lambda+1}(\log P)
\]
Which, upon applying the definiton of $\lambda$ in (\ref{lambdaDef}), is
\[
	I_t(P) \ll P^{t-5}(\log P).
\]
\qed

\section{A Pointwise Minor Arc Bound Sensitive to Multiple Coefficients}\label{sec:Fbound}

Let $\pmb\alpha = (\alpha_1, \ldots, \alpha_k)$ and let
\[
	F_k(\pmb\alpha) = \sum_{n \leq P} \Lambda(n) e(\alpha_1 n + \alpha_2 n^2 + \ldots + \alpha_k n^k).
\]
This section consists of the proof of the following theorem and corollary:

\begin{theorem}\label{Sbound} For $D > 0$, where $D = D(A)$ can be made arbitrarily large by increasing $A$, if $(\alpha_2, \alpha_3) \in \mathfrak{m}$, then
\[
		F_3(\pmb\alpha) \ll P(\log P)^{-D}.
\]
\end{theorem}

\begin{corollary}\label{fiBound} For each $i$, $1\leq i \leq s$,
\[
	\sup_{(\alpha, \beta)\in\mathfrak{m}} f_i(\alpha, \beta) \ll P(\log P)^{-D}.
\]
\end{corollary}

\begin{proof}
Take $\alpha_2 = u_i\alpha$, $\alpha_3 = v_i\beta$, and $\alpha_1 = 0$ in Theorem~\ref{Sbound}, and note that multiplying $\alpha$ and $\beta$ by the integer coefficients $u_i$ and $v_i$ does not move them out of $\mathfrak{m}$ and that there are trivially $\ll P^{1/2}\log P$ prime powers $\leq P$ which contribute $\ll P^{1/2}(\log P)^2$ to the sum.
\end{proof}

We begin by citing some known results on Vinogradov's mean value theorem.  Let
\[
	J_{s,k}(P) = \int_{[0,1)^k} |F_k(\pmb\alpha)|^{2s} d\pmb\alpha.
\]
We cite the bound
\begin{equation}\label{J32bound}
	J_{3,2}(P) \ll P^3 \log P
\end{equation}
from~\cite{rogovskaya} (cf.~\cite{vaughanHLM} chap. 7 exercise 2) and for $s > 6$
\begin{equation}\label{Js3bound}
	J_{s,3}(P) \ll P^{2s-6}
\end{equation}
from equation (7) of~\cite{bdg}.

Let $X = (\log P)^B$ for some $B > 0$ to be fixed later.  For brevity, we let $f(n) := e(\alpha_1n + \alpha_2n^2 + \alpha_3n^3)$.  Then
\[
	F_3(\pmb\alpha) = \sum_{n \leq P} \Lambda(n)f(n).
\]

Applying Vaughan's identity~\cite{vaughansIdentity} to this sum yields

\begin{equation}\label{vaughansIdentityBreakdown}
	\sum_{n \leq P} \Lambda(n)f(n) = S_1 + S_2 + S_3 + S_4,
\end{equation}
where
\[
	S_1 = \sum_{n \leq X} \Lambda(n)f(n),
\]
\[
	S_2 = \sum_{n \leq P} \left( \sum_{\substack{kl=n \\ k \leq X}} \mu(k) \log l \right) f(n),
\]
\[
	S_3 = \sum_{n \leq P} \sum_{\substack{kl=n \\ k \leq X^2}} \left( \sum_{\substack{m,n \\ mn=k \\ m \leq X, n \leq X}} \Lambda(m) \mu(n) \right) f(n),
\]
\[
	S_4 = \sum_{n \leq P} \left( \sum_{\substack{kl=n \\ k>X, l>X}} a(k)b(l) \right) f(n),
\]
with
\[
	a(k) = \sum_{\substack{l | k \\ l>X}} \Lambda(l),
\]
\[
	b(l) = \begin{cases} \mu(l), & l>X \\ 0, & l \leq X. \end{cases}
\]

This now enables us to bound each of the sums $S_1$, $S_2$, $S_3$, $S_4$ individually to obtain the desired bound on $F_3(\pmb\alpha)$.  The bounds on these four sums constitute Lemmas~\ref{S1bound}-\ref{S4lemma}.

\begin{lemma}\label{S1bound}
\begin{equation}
	S_1 \ll X.
\end{equation}
\end{lemma}
\begin{proof}
Since $|f(n)| \ll 1$,
\[
	S_1 = \sum_{n \leq X} \Lambda(n)f(n) \ll \sum_{m \leq X} \Lambda(n) \ll X,
\]
where the last bound is a classical result of Chebyshev.
\end{proof}

\begin{lemma}\label{S3bound}
\[
	S_3 \ll P(\log P)^{B - A/12 + 4}.
\]
\end{lemma}
\begin{proof}

\begin{equation}\label{S3def}
	S_3 = \sum_{n \leq P} \sum_{\substack{kl=n \\ k \leq X^2}} \left( \sum_{\substack{m_1,m_2 \\ m_1m_2=k \\ m_1 \leq X, m_2 \leq X}} \Lambda(m_1) \mu(m_2) \right) f(m_2).
\end{equation}
Let
\[
	c_3(k) := \sum_{\substack{m_1,m_2 \\ m_1m_2=k \\ m_1 \leq X, m_2 \leq X}} \Lambda(m_1) \mu(m_2)
\]
and note for future reference that
\[
	|c_3(k)| \leq \sum_{m|k}\Lambda(m) = \log k.
\]
Interchanging the order of summation in (\ref{S3def}) yields
\[
	S_3 = \sum_{k \leq X^2} c_3(k) \sum_{l \leq P/k} f(kl)
\]
\begin{equation}\label{typeISum}
	= \sum_{k \leq X^2} c_3(k) \sum_{l \leq P/k} e(\alpha_1 kl + \alpha_2 k^2 l^2 + \alpha_3 k^3 l^3).
\end{equation}
We now use Dirichlet's theorem on Diophantine approximation to obtain integers $b_j$, $q_j$ for $j \in \{2,3\}$ such that $(b_j, q_j)=1$,
\begin{equation}\label{alphajApprox}
	|\alpha_j k^j - \frac{b_j}{q_j}| \leq \frac{(\log (P/k))^{A/2}}{q_j(P/k)^j},
\end{equation}
\[
	q_j \leq \frac{(P/k)^j}{(\log (P/k))^{A/2}}.
\]
Assume for contradiction that $q_j \leq (\log (P/k))^{A/2}$ for both $j=2$ and $j=3$ and rewrite (\ref{alphajApprox}) as
\[
	|\alpha_j - \frac{b_j}{k^j q_j}| \leq \frac{(\log (P/k))^{A/2}}{q_j P^j}.
\]
Let $b_j' = b_j/(k^j,b_j)$, $q_j' = k^j q_j/(k^j,b_j)$.  Then
\[
	|\alpha_j - \frac{b_j}{k^j q_j}| \leq \frac{(\log (P/k))^{A/2}}{q_j' P^j},
\]
$(b_j',q_j')=1$, and $q_j' \leq (\log(P/k))^{A/2}$ for $j \in \{2,3\}$.  Let $q = \text{lcm}(q_2',q_3')$ and $a_j = b_j'q/q_j$.  Then  $(a_2,a_3,q)=1$, $q \leq (\log(P/k))^A$, and
\[
	|\alpha_j - \frac{a_j}{q}| \leq \frac{(\log (P/k))^A}{qP^j}.
\]
This implies that $(\alpha_2, \alpha_3) \in \mathfrak{M}$.  However, we have $(\alpha_2, \alpha_3) \in \mathfrak{m}$, which is the desired contradication, so we may assume that $q_j > (\log (P/k))^{A/2}$ for at least one $j_0 \in \{2,3\}$.

We now need a bound on
\[
	H(\pmb\alpha', P/k) := \sum_{l \leq P/k} e(\alpha_1' l + \alpha_2'l^2 + \alpha_3'l^3).
\]
By Theorem 5.2 of~\cite{vaughanHLM}, we have
\[
	H(\pmb\alpha', P/k) \ll (\log P) \left(J_{3,2}(2P) \left(\frac{P}{k}\right)^3 \prod_{j=1}^3\left(\frac{1}{q_j'} + \frac{k}{P} + \frac{q_j'k^j}{P^j} \right) \right)^{1/6}.
\]
Now by~(\ref{J32bound}), we have $J_{3,2}(P) \ll P^3(\log P)$, so
\begin{equation}\label{HaBound}
	H(\pmb\alpha', P/k) \ll \frac{P}{k} (\log P)^2 \prod_{j=1}^3 \left(\frac{1}{q_j'} + \frac{k}{P} + \frac{q_j'k^j}{P^j}\right)^{1/6}.
\end{equation}
Now $\frac{k}{P} \ll P^{-1/2}$, $1/q_j' \ll 1$, and $\frac{q_j'k^j}{P^j} \ll (\log P)^{2jB-A}$, so for $j \neq j_0$
\begin{equation}\label{notj0Bound}
	\frac{1}{q_j'} + \frac{k}{P} + \frac{q_j'k^j}{P^j} \ll 1,
\end{equation}
assuming $2Bj - A < 0$.
For $j = j_0$ we have $1/q_{j_0}' \ll (\log P)^{2Bj_0 - A/2}$, so
\begin{equation}\label{j0Bound}
	\frac{1}{q_j'} + \frac{k}{P} + \frac{q_j'k^j}{P^j} \ll (\log P)^{2Bj_0 - A/2},
\end{equation}
assuming $2Bj_0 - A/2 < 0$.

Applying the bounds of (\ref{notj0Bound}) and (\ref{j0Bound}) to (\ref{HaBound}) yields
\[
	H(\pmb\alpha', P/k) \ll \frac{P}{k} (\log P)^2 (\log P)^{(2Bj_0 - A/2)/12}
\]
\begin{equation}\label{HaABBound}
	\ll \frac{P}{k} (\log P)^{B - A/12 + 2},
\end{equation}
since $j_0 \leq 3$.

Substituting the bound of (\ref{HaABBound}) into (\ref{typeISum}), we obtain
\[
	S_3 \ll \sum_{k \leq X^2} (\log k) \frac{P}{k} (\log P)^{B - A/12 + 2}
\]
\[
	\ll P(\log P)^{B - A/12 + 4}.
\]

\end{proof}

\begin{lemma}\label{S2lemma}
\[
	S_2 \ll P(\log P)^{B-A/12+4}.
\]
\end{lemma}
\begin{proof}
\[
	S_2 = \sum_{n \leq P} \left( \sum_{\substack{kl=n \\ k \leq X}} \mu(k) \log l \right) f(n)
\]
\[
	= \sum_{k \leq X} \mu(k) \sum_{l < P/k} f(kl) \int_1^l \frac{dt}{t}
\]
\[
	= \sum_{k \leq X} \mu(k) \int_1^{P/k} \sum_{l < P/k} f(kl) \frac{dt}{t}
\]
\begin{equation}\label{S2breakdown}
	= \int_1^{P/k} \left(\sum_{k \leq X} \mu(k) \sum_{l < P/k} f(kl) \right) \frac{dt}{t}.
\end{equation}
Now by (\ref{HaABBound}),
\[
	\sum_{l < P/k} f(kl) = H(\pmb\alpha', P/k) \ll \frac{P}{k} (\log P)^{B - A/12 + 2}.
\]
Substituting this into (\ref{S2breakdown}) yields
\[
	S_2 \ll \int_1^{P/k} \sum_{k \leq X} \frac{P}{k} (\log P)^{B - A/12 + 2} \frac{dt}{t}
\]
\[
	\ll P(\log P)^{B - A/12 + 2} \left(\sum_{k \leq X} \frac{\mu(k)}{k}\right) \int_1^{P/k} \frac{dt}{t}
\]
\[
	\ll P(\log P)^{2 + B - A/12} (\log X) (\log P/k)
\]
\[
	\ll P(\log P)^{4 + B - A/12}.
\]
\end{proof}

\begin{lemma}\label{S4lemma}
\[
	S_4 \ll P(\log P)^{4-\min\{A,B\}/(4b^2)}
\]
\end{lemma}
\begin{proof}
We begin by splitting $S_4$ into dyadic ranges.  Let $\mathcal{M} = \{X2^k : 0 \leq k, 2^k \leq P/X^2\}$.  Then
\begin{equation}\label{S4MDef}
	S_4 = \sum_{M \in \mathcal{M}} S_4(M),
\end{equation}
where 
\[
	S_4(M) = \sum_{M < k \leq 2M} \sum_{l \leq P/k} a(k)b(l)f(kl).
\]

Our goal is now to replace the sum over the range $l \leq P/k$ with one over the range $l \leq P/M$.  We begin by considering the integral
\[
	I(x) := \int_\mathbb{R} \frac{\sin(2\pi Rt)}{\pi t}e(-xt)dt,
\]
where $R>0$ is a constant.  Computing the integral via the residue theorem gives
\[
	I(x) = \begin{cases} 1, & |x|<R \\ 0, & |x|>R.\end{cases}
\]
Now for $x \neq R$, $t \geq 1$,
\begin{equation}\label{Tintegral}
	\int_{|t|>T} \frac{\sin(2\pi Rt)}{\pi t}e(-xt)dt = \int_{|t|>T} \frac{e\big((R-x)t\big)-e\big(-(R+x)t\big)}{2\pi it} dt.
\end{equation}
Integrating the right hand side of (\ref{Tintegral}) by parts gives
\[
	\int_{|t|>T} \frac{\sin(2\pi Rt)}{\pi t}e(-xt)dt \ll \frac{1}{T|R-x|} + \frac{1}{T|R+x|} + \frac{1}{T^3} \ll \frac{1}{T\big|R-|x|\big|}.
\]
Thus we can rewrite $I(x)$ as an integral over $[-T,T]$ with an acceptable error term:
\[
	I(x) = \int_{-T}^T \frac{\sin(2\pi Rt)}{\pi t}e(-xt)dt + O\left(\frac{1}{T\big|R-|x|\big|}\right).
\]

We now take $R = \log(\lfloor P\rfloor + \frac{1}{2})$, $x = \log(kl)$, giving us
\[
	S_4(M) = \sum_{M < k \leq 2M} \sum_{l \leq P/M} a(k)b(l)f(kl)I(\log(kl))
\]
\[
	 = \int_{-T}^T \sum_{M < k \leq 2M} \sum_{l \leq P/M} \frac{a(k)b(l)}{(kl)^{2\pi it}} f(kl) \frac{\sin(2\pi Rt)}{\pi t} dt + O\left(\frac{P^2\log P}{T}\right).
\]
Now 
\[
	\frac{\sin(2\pi Rt)}{\pi t} \ll \frac{1}{\pi t} \ll \frac{1}{|t|}
\]
and
\[
	\frac{\sin(2\pi Rt)}{\pi t} \ll \frac{2\pi Rt}{\pi t} \ll R,
\]
so
\[
	\frac{\sin(2\pi Rt)}{\pi t} \ll \min(R, 1/|t|).
\]
Take $T = P^3$, $a(k,t) = a(k)k^{-2\pi it}$, $b(l,t) = b(l)l^{-2\pi it}$, and let
\begin{equation}\label{S4MtDef}
	S_4(M, t) = \sum_{M < k \leq 2M} \sum_{l \leq P/k} a(k, t)b(l, t)f(kl).
\end{equation}
Then
\[
	S_4(M) \ll \sup_{|t|<T}|S_4(M,t)| \int_{-T}^T \frac{\sin(2\pi Rt)}{\pi t} dt
\]
\[
	\ll 1 + (\log P)\sup_{|t|<T}|S_4(M,t)|.
\]

We now consider $S_4(M,t)$.  Let $b>6$.  By H\"older's inequality
\begin{equation}\label{procedureStart}
	S_4(M,t)^{2b} \ll \left(\sum_{M < k \leq 2M} |a(k,t)|^\frac{2b}{2b-1} \right)^{2b-1} \sum_{M < k \leq 2M} \left| \sum_{l \leq P/M} b(l,t)f(kl) \right|^{2b}.
\end{equation}
Now $|a(k,t)| = |a(k)| \leq \log k \ll \log M \ll \log P$, so
\[
	S_4(M,t)^{2b} \ll \left( M(\log P)^\frac{2b}{2b-1}\right)^{2b-1} \sum_{M < k \leq 2M} \left| \sum_{l \leq P/M} b(l,t)f(kl) \right|^{2b}
\]
\begin{equation}\label{2b}
	\ll (\log P)^{2b} M^{2b-1} \sum_{M < k \leq 2M} \left| \sum_{l \leq P/M} b(l,t)f(kl) \right|^{2b}.
\end{equation}
Expanding the $2b$-th power in (\ref{2b}) yields
\[
	\left| \sum_{l \leq P/M} b(l,t)f(kl) \right|^{2b}
\]
\begin{equation}\label{introsl}
	= \sum_{\substack{\mathbf l \\ l_j \leq P/M}} \left( \prod_{i=1}^b b(l_i,t) \prod_{i=b+1}^{2b} \overline{b(l_i,t)} \right) e\left(\alpha_1ks_1(\mathbf l) + \alpha_2 k^2 s_2(\mathbf l) + \alpha_3 k^3 s_3(\mathbf l)\right)
\end{equation}
where
\[
	s_j(\mathbf l) = l_1^j + \ldots + l_b^j - l_{b+1}^j - \ldots - l_{2b}^j.
\]
Collecting terms in (\ref{introsl}) by values of $s_j$ yields
\begin{equation}\label{bfSumBound}
	\left| \sum_{l \leq P/M} b(l,t)f(kl) \right|^{2b} = \sum_{\substack{\mathbf{v} \\ |v_j| \leq bP^j}} R_1(\mathbf{v}) e(\alpha_1kv_1 + \alpha_2k^2v_2 + \alpha_3k^3v_3)
\end{equation}
where
\[
	R_1(\mathbf{v}) = \sum_{\substack{\mathbf{l} \\ l_j \leq P/M \\ \mathbf{s}(\mathbf{l}) = \mathbf{v}}} \prod_{i=1}^b b(l_i,t) \prod_{i=b+1}^{2b} \overline{b(l_i,t)} \ll J_{b,3}(P/M) \ll (P/M)^{2b-6}
\]
by~(\ref{Js3bound}).
Substituting (\ref{bfSumBound}) into (\ref{2b}) yields
\[
	S_4(M,t)^{2b} \ll (\log P)^{2b}M^{2b-1} \sum_{\substack{\mathbf{v} \\ |v_j| \leq bP^jM^{-j}}} R_1(\mathbf{v}) \sum_{M < k \leq 2M} e(\alpha_1kv_1 + \alpha_2k^2v_2 + \alpha_3k^3v_3)
\]
\begin{equation}\label{procedureEnd}
	\ll (\log P)^{2b}M^5 P^{2b-6} \sum_{\substack{\mathbf{v} \\ |v_j| \leq bP^jM^{-j}}} \sum_{M < k \leq 2M} e(\alpha_1kv_1 + \alpha_2k^2v_2 + \alpha_3k^3v_3).
\end{equation}
We now repeat the procedure followed from (\ref{procedureStart}) to (\ref{procedureEnd}).  By H\"older's inequality
\begin{multline}
	S_4(M,t)|^{4b^2} \ll \left( (\log P)^{2b} M^5 P^{2b-6} \right)^{2b} \left( \sum_{\substack{\mathbf{v} \\ |v_j| \leq bP^jM^{-j}}} 1^\frac{2b}{2b-1} \right)^{2b-1} \\ \times \sum_{\substack{\mathbf{v} \\ |v_j| \leq bP^jM^{-j}}} \left| \sum_{M < k \leq 2M} e(\alpha_1kv_1 + \alpha_2k^2v_2 + \alpha_3k^3v_3) \right|^{2b}
\end{multline}
\[
	\ll (\log P)^{4b}M^{10b}P^{4b^2-12b} \left( b^3P^6M^{-6} \right)^{2b-1}\sum_{\substack{\mathbf{v} \\ |v_j| \leq bP^jM^{-j}}} \left| \sum_{M < k \leq 2M} e(\alpha_1kv_1 + \alpha_2k^2v_2 + \alpha_3k^3v_3) \right|^{2b}
\]
\begin{equation}\label{orNot2b}
	\ll (\log P)^{4b^2} M^{6-2b} P^{4b^2-6} \sum_{\substack{\mathbf{v} \\ |v_j| \leq bP^jM^{-j}}} \left| \sum_{M < k \leq 2M} e(\alpha_1kv_1 + \alpha_2k^2v_2 + \alpha_3k^3v_3) \right|^{2b}.
\end{equation}
We expand the $2b$-th power in (\ref{orNot2b}) and collect like terms. Thus
\[
	\left| \sum_{M < k \leq 2M} e(\alpha_1kv_1 + \alpha_2k^2v_2 + \alpha_3k^3v_3) \right|^{2b}
\]
\[
	= \sum_{\substack{\mathbf{k} \\ M < k_j \leq 2M}} e(\alpha_1s_1(\mathbf{k})v_1 + \alpha_2s_2(\mathbf{k})v_2 + \alpha_3s_3(\mathbf{k})v_3)
\]
\begin{equation}\label{eSumBound}
	= \sum_{\substack{\mathbf{u} \\ |u_j| \leq b2^jM^j}} R_2(\mathbf{u}) e(\alpha_1u_1v_1 + \alpha_2u_2v_2 + \alpha_3u_3v_3)
\end{equation}
where
\[
	R_2(\mathbf{u}) = \sum_{\substack{\mathbf{k} \\ M < k_j \leq 2M \\ \mathbf{s}(\mathbf{k}) = \mathbf{u}}} 1 \ll J_{b,3}(2M) \ll M^{2b-6}
\]
by~(\ref{Js3bound}).
Substituting (\ref{eSumBound}) into (\ref{orNot2b}), we obtain
\[
	S_4(M,t)^{4b^2} \ll (\log P)^{4b^2} P^{4b^2-6} \sum_{\substack{\mathbf{u} \\ |u_j| \leq b2^jM^j}} \left| \sum_{\substack{\mathbf{v} \\ |v_j| \leq bP^jM^{-j}}} e(\alpha_1u_1v_1 + \alpha_2u_2v_2 + \alpha_3u_3v_3) \right|.
\]
Summing over each of the $v_j$ gives
\[
	S_4(M,t)^{4b^2} \ll (\log P)^{4b^2} P^{4b^2-6} \sum_{\substack{\mathbf{u} \\ |u_j| \leq b2^jM^j}} \prod_{j=1}^3 \min\left(\frac{P^j}{M^j},\frac{1}{\|\alpha_ju_j\|}\right).
\]
Applying Lemma 2.2 of~\cite{vaughanHLM} yields
\begin{equation}\label{S4MtBound}
	S_4(M,t)^{4b^2} \ll (\log P)^{4b^2+3} P^{4b^2} \prod_{j=1}^3 \left(\frac{1}{q_j} + \frac{1}{M^j} + \frac{M^j}{P^j} + \frac{q_j}{P^j}\right).
\end{equation}
Combining (\ref{S4MtBound}) with (\ref{S4MDef}) and (\ref{S4MtDef}), we obtain
\[
	S_4 \ll P(\log P)^4 \prod_{j=1}^3 \left(\frac{1}{q_j} + \frac{1}{X^j} + \frac{q_j}{P^j}\right)^{1/(4b^2)}.
\]
Recalling that $q_j > (\log P)^A$ for some $j$ and $X = (\log P)^B$, this is
\begin{equation}\label{S4bound}
	S_4 \ll P (\log P)^{4-\min(A,B)/(4b^2)}
\end{equation}
for $b > 6$.
\end{proof}

\textit{Proof of Theorem~\ref{Sbound}}: Using the Vaughan's identity breakdown of (\ref{vaughansIdentityBreakdown}) and the estimates for the $S_i$ found in Lemmas~\ref{S1bound},~\ref{S3bound},~\ref{S2lemma}, and~\ref{S4lemma}, we have
\[
	F_3(\pmb\alpha) = S_1 + S_2 + S_3 + S_4
\]
\[
	\ll (\log P)^B + P(\log P)^{B - A/12 + 4} + P(\log P)^{B-A/12+4} + P (\log P)^{4-\min(A,B)/(4b^2)}.
\]
So, taking $B > 4b^2D(D+4)$ and $A > 12(B + D + 4)$ for some $D > 0$ yields
\[
	F_3(\pmb\alpha) \ll P(\log P)^{-D}.
\]
\qed

\section{Major Arc Approximations}\label{sec:majorArcs}

On a typical major arc $\mathfrak{M}(a,b,q)$, let $\alpha = \frac{a}{q}+\theta$, $\beta = \frac{b}{q}+\omega$, with $\theta < \frac{Q}{qP^2}$, $\omega < \frac{Q}{qP^3}$, and $q < Q$.  For ease of notation, let $\frac{Q}{qP^2} = \Theta$, $\frac{Q}{qP^3} = \Omega$.  Let
\[
	W_i(q,a,b) = \sum_{\substack{r=1 \\ (r,q)=1}}^q e\left(\frac{au_ir^2 + bv_ir^3}{q}\right),
\]
\[
	f_i^*(\alpha, \beta) = \frac{1}{\phi(q)} W_i(q,a,b) \int_0^P e(\theta u_i x^2 + \omega v_i x^3) dx,
\]
\[
	T_i(x,a,b) = \sum_{p \leq x}(\log p) e\left(\frac{au_ip^2 + bv_ip^3}{q}\right),
\]
and for $x > \sqrt{P}$,
\[
	T_i^*(x,a,b) = \sum_{\sqrt{P} < p \leq x}(\log p) e\left(\frac{au_ip^2 + bv_ip^3}{q}\right).
\]

We begin with preliminary bounds on $T_i(x,a,b)$ and $T_i^*(q,a,b)$.
\begin{lemma}\label{SWLemma}
\[
	T_i(x,a,b) = \frac{x}{\phi(q)}W_i(q,a,b) + O(x\exp(-C(\log x)^{1/2})).
\]
\end{lemma}
\begin{proof}
The exponential function $e((au_ip^2 + bv_ip^3)/q)$ is only sensitive to the residue class of $p$ modulo $q$, so
\[
	T_i(x,a,b) = \sum_{\substack{r=1 \\ (r,q)=1}}^q\sum_{\substack{p \leq x \\ p \equiv r \pmod* q}} (\log p) e\left(\frac{au_ir^2 + bv_ir^3}{q}\right) + O(q^\varepsilon \log q)
\]
\[
	= \sum_{\substack{r = 1\\(r,q)=1}}^q \left(e\left(\frac{au_ir^2 + bv_ir^3}{q}\right) \sum_{\substack{p \leq x\\p\equiv r\pmod* q}} \log p \right) + O(q^\varepsilon \log q).
\]
Now by the Siegel-Walfisz theorem we have that
\begin{equation}\label{siegelWalfisz}
	\sum_{\substack{p \leq x\\p\equiv r\pmod*{q}}} \log p = \frac{x}{\phi(q)} + O(x\exp(-C(\log x)^{1/2})),
\end{equation}
so  
\[
	T_i(x,a,b) = \sum_{\substack{r = 1\\(r,q)=1}}^q \left(e\left(\frac{au_ir^2 + bv_ir^3}{q}\right)\left(\frac{x}{\phi(q)} + O(x\exp(-C(\log x)^{1/2}))\right)\right)
\]
\[
	= \frac{x}{\phi(q)} W_i(q,a,b) + W_i(q,a,b)\left(O(x\exp(-C(\log x)^{1/2}))\right)
\]
\[
	= \frac{x}{\phi(q)}W_i(q,a,b) + O(x\exp(-C(\log x)^{1/2})).
\]
\end{proof}

\begin{corollary}\label{TStarBound} For $x > \sqrt{P}$,
\[
	T_i^*(x,a,b) = \frac{x}{\phi(q)}W_i(q,a,b) + O(x\exp(-C(\log x)^{1/2})).
\]
\end{corollary}
\begin{proof}
\[
	T_i^*(x,a,b) = T_i(x,a,b) - \sum_{p \leq \sqrt{P}}(\log p) e\left(\frac{au_ip^2 + bv_ip^3}{q}\right)
\]
\[
	= \frac{x}{\phi(q)}W_i(q,a,b) + O(x\exp(-C(\log x)^{1/2})) + O(P^{1/2})
\]
\[
	\frac{x}{\phi(q)}W_i(q,a,b) + O(x\exp(-C(\log x)^{1/2})).
\]
\end{proof}

\begin{lemma}\label{majorApprox}
On $\mathfrak{M}(q,a,b)$,
\[
	f_i(\alpha, \beta) = f_i^*(\alpha, \beta) + O(P\exp(-C(\log P)^{1/2}))
\]

for some positive constant $C$.
\end{lemma}

\begin{proof}
First, we isolate the range $(\sqrt{P}, P]$, bounding the remainder immediately.
\[
	|f_i(\alpha, \beta) - f_i^*(\alpha, \beta)|
\]
\[
	= \left|\sum_{p \leq P} (\log p)e(\alpha u_i p^2 + \beta v_i p^3) - \frac{1}{\phi(q)}W_i(q,a,b)\int_0^P e(\theta u_i x^2 + \omega v_i x^3) dx \right|
\]
\begin{multline}\nonumber
	= \left|\sum_{\sqrt{P} < p \leq P} (\log p)e(\alpha u_i p^2 + \beta v_i p^3) -  \frac{1}{\phi(q)}W_i(q,a,b)\int_{\sqrt{P}}^P e(\theta u_i x^2 + \omega v_i x^3) dx \right| \\
	+ O(P^{1/2}\log P).
\end{multline}

Now
\[
	\left|\sum_{\sqrt{P} < p \leq P} (\log p)e(\alpha u_i p^2 + \beta v_i p^3) - \frac{1}{\phi(q)}W_i(q,a,b)\int_{\sqrt{P}}^P e(\theta u_i x^2 + \omega v_i x^3) dx \right|
\]

\begin{multline}
	= \Bigg| W_i(q,a,b)  \sum_{\substack{\sqrt{P} < p \leq P \\ p \equiv r \pmod* q}} (\log p)e(\theta u_i p^2 + \omega v_i p^3)  \\ - \frac{1}{\phi(q)}W_i(q,a,b)\int_{\sqrt{P}}^P e(\theta u_i x^2 + \omega v_i x^3) dx \Bigg|
\end{multline}

\begin{multline}\label{needsAbel}
	= \sum_{\sqrt{P} < m \leq P}  \bigg[ (\log m)e\left(\frac{au_im^2+bv_im^3}{q}\right) \mathbb{1}_\mathcal{P} \\ - \frac{1}{\phi(q)} W_i(q,a,b) \bigg] e(\theta u_i m^2 + \omega v_i m^3),
\end{multline}
where $\mathbb{1}_\mathcal{P}$ is the indicator function of the primes.  

We now apply Abel summation to (\ref{needsAbel}), with the term in square brackets serving as the coefficient.  This yields that
\[
	|f_i(\alpha, \beta) - f_i^*(\alpha, \beta)|
\]
\begin{equation}\label{needsSW}
\begin{split}
	 =~ & e(\theta u_i P^2 + \omega v_i P^3) \bigg( T_i(x,a,b) - \frac{1}{\phi(q)}\sum_{\sqrt{P} < m \leq P}W_i(q,a,b)\bigg) \\ & -\int_{\sqrt{P}}^P 2\pi i(2\theta u_i x + 3\omega v_i x^2) \bigg(T_i(x,a,b) - \frac{1}{\phi(q)}\sum_{\sqrt{P} < m \leq x}W_i(q,a,b) \bigg) dx.
\end{split}
\end{equation}

Now Lemma~\ref{TStarBound} gives that for $x \geq \sqrt{P}$,
\[
	T_i^*(x,a,b) - \frac{1}{\phi(q)}\sum_{\sqrt{P} < m \leq x}W_i(q,a,b) \ll x\exp(-C(\log x)^{1/2}),
\]

so
\[
	|f_i(\alpha, \beta) - f_i^*(\alpha, \beta)|
\]

\begin{equation}\nonumber
\begin{split}
	=~ & e(\theta u_i P^2 + \omega v_i P^3) \left(O(\phi(q) P \exp(-C(\log P)^{1/2}))\right) 
	\\ & - 2\pi i\int_0^P (2\theta u_i x + 3\omega v_i x^2) \left(O(\phi(q) x \exp(-C(\log P)^{1/2}))  \right) dx
	\\ & + O(P^{1/2}\log P)
\end{split}
\end{equation}

\[
	\ll (1 + |\theta|P^2 + |\omega|P^3) \phi(q) P \exp(-C(\log P)^{1/2})
\]

\[
	\ll (\log P)^A \frac{\phi(q)}{q} P \exp(-C(\log P)^{1/2})
\]

\[
	\ll P \exp(-C(\log P)^{1/2}).
\]

\end{proof}

For clarity of notation, let
\[
	A(q) = \mathop{\sum_{a=1}^q \sum_{b=1}^q}_{(a,b,q)=1} \frac{1}{\phi(q)^s} \prod_{i=1}^s W_i(q,a,b),
\]
\[
	\mathfrak{S}(Q) = \sum_{q<Q} A(q),
\]
\[
	J(Q) = \int_{|\theta|<Q/P^2} \int_{|\omega|<Q/P^3} \prod_{i=1}^s \int_0^P  e(\theta u_i x^2 + \omega v_i x^3) dx d\omega d\theta.
\]

We are now able to state the primary lemma of this section:

\begin{lemma}\label{singularDecomp}
\[
	R(P) = \mathfrak{S}(Q)J(Q) + O(P^{s-5}(\log P)^{-E}).
\]
\end{lemma}
\begin{proof}

We first introduce the inhomogeneous major arcs 
\[
	\mathfrak{B}(q, \mathbf{r}, Q) = \{(a,b) : |\alpha - a/q| < \frac{Q}{P^2}, |\beta - b/q| < \frac{Q}{P^3}\}
\]
for $1 \leq Q \leq P$, $q < Q$, $1 \leq a \leq q$, $1 \leq b \leq q$, and $(a, b, q) = 1$.  Note that $\mathfrak{M} \subseteq \mathfrak{B}$ and thus $\mathfrak{B} \setminus \mathfrak{M} \subseteq \mathfrak{m}$.

It follows immediately from Lemma~\ref{majorApprox} that
\begin{equation}\label{fProdBound}
	\left|\prod_{i=1}^s f_i(\alpha, \beta) - \prod_{i=1}^s f_i^*(\alpha, \beta)\right| \ll P^s\exp(-C(\log P)^{1/2}).
\end{equation}
Summing (\ref{fProdBound}) over all inhomogeneous major arcs gives

\[
	\int_\mathfrak{B} \left|\prod_{i=1}^s f_i(\alpha, \beta) - \prod_{i=1}^s f_i^*(\alpha, \beta)\right| d\alpha d\beta
\]
\[
	 = \sum_{q<Q} \mathop{\sum_{a=1}^q \sum_{b=1}^q}_{(a,b,q)=1} \int_{\mathfrak{B}(a,b,q)} \left|\prod_{i=1}^s f_i(\alpha, \beta) - \prod_{i=1}^s f_i^*(\alpha, \beta)\right| d\alpha d\beta
\]
\[
 \ll \sum_{q<Q} \mathop{\sum_{a=1}^q \sum_{b=1}^q}_{(a,b,q)=1} \int_{-Q/P^2}^{Q/P^2} \int_{-Q/P^3}^{Q/P^3} P^s\exp(-C(\log P)^{1/2}) d\alpha d\beta
\]
\begin{equation}\label{f-f*Bound}
 \ll Q^3P^{s-5}\exp(-C(\log P)^{1/2}).
\end{equation}
We now wish to compute
\[
	\int_\mathfrak{B} \prod_{i=1}^s f_i^*(\alpha, \beta) d\alpha d\beta
\]
\[
	 = \sum_{q<Q} \mathop{\sum_{a=1}^q \sum_{b=1}^q}_{(a,b,q)=1}  \prod_{i=1}^s \frac{1}{\phi(q)} W_i(q,a,b) \int_{-Q/P^2}^{Q/P^2} \int_{-Q/P^3}^{Q/P^3} \int_0^P e(\theta u_i x^2 + \omega v_i x^3) dx d\theta d\omega + O(P^{s-5}(\log P)^{-E})
\]
\begin{equation}\label{ugly}
	 = \mathfrak{S}(Q) J(Q) + O(P^{s-5}(\log P)^{-E}).
\end{equation}
Combining (\ref{f-f*Bound}) and (\ref{ugly}) yields the inhomogeneous major arc bound
\begin{equation}\label{majorArcBreakdown}
	\int_\mathfrak{B} \prod_{i=1}^s f_i(\alpha, \beta) = \mathfrak{S}(Q) J(Q) + O(P^{s-5} (\log P)^{-E}).
\end{equation}
Combining Corollary~\ref{huaBound} and Corollary~\ref{fiBound} yields the minor arc bound
\begin{equation}\label{minorArcBound}
	\int_\mathfrak{m} \prod_{i=1}^s f_i(\alpha, \beta) \ll P^{s-5} (\log P)^{-E},
\end{equation}
and moreover, since $\mathcal{A}\setminus\mathfrak{B} \subseteq \mathfrak{m}$, by Corollary~\ref{fiBound} and Theorem~\ref{huaBound} we have
\begin{equation}\label{inhomogeneousMinorArcBound}
	\int_{\mathcal{A}\setminus\mathfrak{B}} \prod_{i=1}^s f_i(\alpha, \beta) \ll P^{s-5} (\log P)^{-E}.
\end{equation}
Now by (\ref{Rdef}), (\ref{majorArcBreakdown}), and (\ref{inhomogeneousMinorArcBound}) we have
\begin{equation}\label{RBreakdown}
	R(P) = \mathfrak{S}(Q)J(Q) + O(P^{s-5} (\log P)^{-E}).
\end{equation}

\end{proof}

\section{Convergence of the Singular Series}\label{sec:singularSeriesConvergence}

\begin{lemma}
Let $(q_1,q_2)=1$.  Then
\[
	W_i(q_1q_2,a,b) = W_i(q_2,aq_1,bq_1^2)W_i(q_1,aq_2,bq_2^2).
\]
\end{lemma}
\begin{proof}
Each residue class $r$ modulo $q_1q_2$ with $(r, q_1q_2) = 1$ is uniquely represented as $cq_1+dq_2$ with $1 \leq c \leq q_2$, $(c, q_2) = 1$, $1 \leq d \leq q_1$, and $(d, q_1) = 1$.  Also, $cq_1$, $dq_2$ run over all residue classes modulo $q_2$, $q_1$ with $(cq_1, q_2) = 1$, $(dq_2, q_1) = 1$ respectively.  Thus
\[
	W_i(q_1q_2,a,b) = \sum_{\substack{c=1 \\ (c,q_2)=1}}^{q_2} \sum_{\substack{d=1 \\ (d,q_1)=1}}^{q_1} e\left( \frac{au_i(cq_1+dq_2)^2 + bv_i(cq_1+dq_2)^3}{q_1q_2} \right)
\]
\[
	= \sum_{\substack{c=1 \\ (c,q_2)=1}}^{q_2} \sum_{\substack{d=1 \\ (d,q_1)=1}}^{q_1} e\left( \frac{au_ic^2q_1+bv_ic^3q_2^2}{q_2} \right) e\left( \frac{au_id^2q_2 + bv_id^3q_2^2}{q_1q_2} \right)
\]
\[
	= W_i(q_2,aq_1,bq_1^2)W_i(q_1,aq_2,bq_2^2).
\]
\end{proof}

\begin{lemma}
A(q) is multiplicative.
\end{lemma}
\begin{proof}
Let $(q_1,q_2)=1$.  Then
\[
	A(q_1q_2) = \mathop{\sum_{a=1}^{q_1q_2}\sum_{b=1}^{q_1q_2}}_{(a,b,q_1q_2)=1} \frac{1}{\phi(q_1q_2)^s} \prod_{i=1}^s W_i(q_1q_2,a,b).
\]
Now $a$ and $b$ can be represented by $a_1q_2+a_2q_1$ and $b_1q_2+b_2q_1$ respectively, with $1 \leq a_1,b_1 \leq q_1$, $1 \leq a_2,b_2 \leq q_2$.  So we can rewrite our sum as
\[
	A(q_1q_2) = \mathop{\sum_{a_1=1}^{q_1}\sum_{b_1=1}^{q_1}}_{(a_1,b_1,q_1)=1} \mathop{\sum_{a_2=1}^{q_2}\sum_{b_2=1}^{q_2}}_{(a_2,b_2,q_2)=1} \frac{1}{\phi(q_1q_2)^s} \prod_{i=1}^s W_i(q_2,a_2q_1^2,b_2q_1^3) W_i(q_1,a_1q_2^2,b_1q_2^3).
\]
Now since $(q_1,q_2)=1$, $(a_2,b_2,q_1)=1$, and $(a_1,b_1,q_2)=1$, we have that $a_2q_1^2,b_2q_1^3,a_1q_2^2,b_1q_2^3$ run through complete sets of residue classes modulo $q_2,q_2,q_1,q_1$ respectively.  Thus
\[
	A(q_1q_2) = \mathop{\sum_{a_1=1}^{q_2}\sum_{b_1=1}^{q_2}}_{(a_1,b_1,q_2)=1} \mathop{\sum_{a_2=1}^{q_1}\sum_{b_2=1}^{q_1}}_{(a_2,b_2,q_1)=1} \frac{1}{\phi(q_1q_2)^s} \prod_{i=1}^s W_i(q_2,a_2,b_2) W_i(q_1,a_1,b_1)
\]
\[
	= A(q_1)A(q_2).
\]
\end{proof}

Let $\mathfrak{S}$ be the completed singular series
\[
	\mathfrak{S} = \sum_{q=1}^\infty A(q).
\]

Since $A(q)$ is multiplicative,
\begin{equation}\label{Sproduct}
	\mathfrak{S} = \prod_p \left(1 + \sum_{k=1}^\infty A(p^k)\right).
\end{equation}

\begin{lemma}\label{frakSConverges} $\mathfrak{S}$ converges absolutely.
\end{lemma}
\begin{proof}

\[
	A(p^k) = \mathop{\sum_{a=1}^{p^k}\sum_{b=1}^{p^k}}_{(a,b,p^k)=1} \frac{1}{\phi(p^k)^s} \prod_{i=1}^s W_i(p^k,a,b).
\]
By Lemma~\ref{WqaBound} and the fact that there are $\ll p^{2k}$ choices for the pair $a$, $b$, we have
\[
	A(p^k) \ll p^{2k} \phi(p^k)^{-s} ((p^k)^{\frac{1}{2}+\varepsilon})^s 
\]
\[
	\ll (p^k)^{2-\frac{1}{2}s+\varepsilon}.
\]
Since $s \geq 7$, we have
\begin{equation}\label{Abound}
	A(p^k) \ll (p^k)^{-\frac{3}{2}+\varepsilon}.
\end{equation}
Thus
\[
	\sum_{k=1}^\infty A(p^k) \ll \sum_{k=1}^\infty (p^k)^{-\frac{3}{2}+\varepsilon} = \frac{p^{-3/2+\varepsilon}}{1-p^{-3/2+\varepsilon}} \ll p^{-3/2+\varepsilon}.
\]
Then
\[
	\sum_p \sum_{k=1}^\infty A(p^k) \ll \sum_p p^{-3/2+\varepsilon}
\]
converges, so
\[
	\mathfrak{S} = \prod_p \left(1 + \sum_{k=1}^\infty A(p^k)\right)
\]
converges.

\end{proof}

\section{Positivity of the Singular Series}\label{sec:singularSeriesPositivity}

To show that $R(P)$ is eventually positive, we now need to show that $\mathfrak{S}$ is positive.

\begin{lemma}\label{upperProductBound} There exists $R > 0$ such that
\[
	\frac{1}{2} < \prod_{p \geq R} \left(1 + \sum_{k=1}^\infty A(p^k)\right).
\]
\end{lemma}
\begin{proof}
By~(\ref{Abound}), we have $A(p^k) \ll (p^k)^{-3/2+\varepsilon} \ll (p^k)^{-1/4}$.  Choose $C, R$ such that $A(p^k) \leq Cp^{-5/4} < Cp^{-1/4} < \frac{1}{8}$ for all $p \geq R-1$.  Then
\[
	\prod_{p\geq R}\left(1-Cp^{-5/4}\right) \geq 1 - \sum_{p\geq R} Cp^{-5/4}
\]
\[
	\geq 1 - C\int_{R-1}^\infty x^{-5/4} dx = 1 - 4C(R-1)^{-1/4} \geq \frac{1}{2}.
\]

\end{proof}

We now need only show that for $p \leq R$, $1 + \sum_{k=1}^\infty A(p^k) > 0$.  Define $M(q)$ to be the number of solutions $(x_1, \ldots, x_s)$ to the simultaneous equations
\begin{equation}\nonumber
	\begin{split}
		& \sum_{i=1}^s u_ix_i^2 \equiv 0 \pmod q \\
		& \sum_{i=1}^s v_ix_i^3 \equiv 0 \pmod q
	\end{split}
\end{equation}
with $(x_i,q)=1$ for all $i$.

\begin{lemma}\label{AMrelation} For any positive integer $q$,
\[
	M(q) = \frac{\phi(q)^s}{q^2} \sum_{d|q} A(d).
\]
\end{lemma}
\begin{proof}
\[
	M(q) = \frac{1}{q^2} \sum_{r_1=1}^q \sum_{r_2=1}^q \sum_{\substack{x_1=1 \\ (x_1,q)=1}}^q \cdots \sum_{\substack{x_s=1 \\ (x_s,q)=1}}^q e\left(\frac{r_1(u_1x_1^2 + \ldots + u_sx_s^2) + r_2(v_1x_1^3 + \ldots + v_sx_s^3)}{q}\right)
\]
\[
	= \frac{1}{q^2} \sum_{r_1=1}^q \sum_{r_2=1}^q \prod_{i=1}^s \sum_{\substack{x_i=1 \\ (x_1,q)=1}}^q e\left(\frac{r_1u_ix_i^2 + r_2v_ix_i^3}{q}\right).
\]
Let $d = \frac{q}{(r_1,r_2,q)}$, $a_1 = \frac{r_1}{(r_1,r_2,q)}$, and $a_2 = \frac{r_2}{(r_1,r_2,q)}$.  Then, rearranging according to the value of $d$, we have
\[
	M(q) = \frac{1}{q^2} \sum_{d|q} \mathop{\sum_{a_1=1}^d \sum_{a_2=1}^d}_{(a_1,a_2,d)=1} \prod_{i=1}^s \frac{\phi(q)}{\phi(d)} \sum_{\substack{x_i=1 \\ (x_i,d)=1}}^d e\left(\frac{a_1u_ix_i^2 + a_2v_ix_i^3}{d}\right)
\]
\[
	= \frac{\phi(q)^s}{q^2} \sum_{d|q} A(d).
\]

\end{proof}

\begin{lemma}\label{Mlift} For positive integers $t, \gamma$ with $t > \gamma$,
\[
	M(p^t) \geq M(p^\gamma)p^{(t-\gamma)(s-2)}.
\]
\end{lemma}
\begin{proof}

This is \cite{wooleySAEii}, Lemma 6.7, with the added observation that 
\[
	\max\{|b_1-a_1|_p, |b_2-a_2|_p\} \leq p^{-\gamma} \Rightarrow p^{\gamma} | (b_1-a_1), (b_2-a_2).
\]   
So if $a_1,b_1 \not\equiv 0 $ (mod $p$), then $a_2,b_2 \not\equiv 0 $ (mod $p$).  Thus the argument lifts solutions over reduced residue classes modulo $p^\gamma$ to solutions over reduced residue classes modulo $p^t$, so it applies here without modification.
\end{proof}

\begin{theorem}\label{Spositivity}	If for every prime $p$ there exists a positive integer $\gamma$ such that $M(p^\gamma) > 0$, then $\mathfrak{S} > 0$.
\end{theorem}
\begin{proof}
By Lemma~\ref{AMrelation},
\[
	1 + \sum_{k=1}^\infty A(p^k) = \lim_{t\rightarrow\infty}\frac{p^{2t}}{\phi(p^t)^s}M(p^t)
\]
\[
	\geq \lim_{t\rightarrow\infty}p^{(2-s)t}M(p^t).
\]
By Lemma~\ref{Mlift}, for some positive integer $\gamma$,
\[
	1 + \sum_{k=1}^\infty A(p^k) \geq \lim_{t\rightarrow\infty}p^{(2-s)t}M(p^\gamma)p^{(t-\gamma)(s-2)}
\]
\begin{equation}\label{lowerProductBound}
	\geq \lim_{t\rightarrow\infty} p^{(-\gamma)(s-2)} > 0.
\end{equation}
The lemma now follows from (\ref{Sproduct}), Lemma \ref{upperProductBound}, and (\ref{lowerProductBound}).
\end{proof}

In Sections~\ref{sec:localConditions} and~\ref{sec:computationalTechniques} we prove that, under the conditions of Theorem~\ref{mainThmv2}, for every $p$ there exists a positive integer $\gamma$ such that $M(p^\gamma)>0$.

\section{Solvability of the Local Problem}\label{sec:localConditions}

We now consider the local system
\begin{equation}
	\begin{split}
		u_1x_1^2 + \ldots + u_sx_s^2 = 0 \pmod p\\
		v_1x_1^3 + \ldots + v_sx_s^3 = 0 \pmod p
	\end{split}
\end{equation}
with $x_i \neq 0$ in $\mathbb{Z}/p\mathbb{Z}$.

We will prove the following result:
\begin{theorem}\label{localSolvability} The system
\begin{equation}
	\begin{split}
		u_1x_1^2 + \ldots + u_sx_s^2 = U \pmod p\\
		v_1x_1^3 + \ldots + v_sx_s^3 = V \pmod p
	\end{split}
\end{equation}
has a solution $(x_1, \ldots, x_s)$ with all $x_i \neq 0$ modulo every prime $p$ if
\begin{enumerate}

\item $\displaystyle\sum_{i=1}^s u_i = U \pmod 2$ and $\displaystyle\sum_{i=1}^s v_i = V \pmod 2,$

\item $\displaystyle\sum_{i=1}^s u_i = U \pmod 3,$ and

\item for each prime $p$ at least 7 of the $u_i$, $v_i$ are not zero modulo $p$.

\end{enumerate}
\end{theorem}

Observe that if the system
\begin{equation}
	\begin{split}
		u_1x_1^2 + \ldots + u_tx_t^2 = U \pmod p\\
		v_1x_1^3 + \ldots + v_tx_t^3 = V \pmod p
	\end{split}
\end{equation}
has a solution for all $u_1, \ldots, u_t, v_1, \ldots, v_t \neq 0$, then so does the system 
\begin{equation}
	\begin{split}
		u_{i_1}x_{i_1}^2 + \ldots + u_{i_t}x_{i_t}^2 = U \pmod p\\
		v_{j_1}x_{j_1}^3 + \ldots + v_{j_t}x_{j_t}^3 = V \pmod p
	\end{split}
\end{equation}
for any $\{i_1, \ldots, i_t\}$, $\{j_1, \ldots, j_t\} \subset \{1, \ldots, s\}$.  Also observe that the conditions of Theorem~\ref{localSolvability} guarantee solvability modulo $p=2$ and $p=3$: $p=2$ is immediate and for $p=3$, the condition guarantees that the quadratic equation is satisfied and each term $v_ix_i^3$ of the cubic equation can be independently set to $1$ or $-1$, allowing us to set $v_1x_1^3 = V$ if $V \not\equiv 0 \pmod 3$ and partition the remainder of $\{1, \ldots, t\}$ into groups of 2 and 3, which can be zeroed by setting them to $\{1, -1\}$ and $\{1, 1, 1\}$.

Thus we have reduced Theorem~\ref{localSolvability} to this lemma:
\begin{lemma}\label{localTLemma} For all $u_i$, $v_i \neq 0 \pmod p$, $p \geq 5$, $t \geq 7$, $U$, $V$, there exist $\{x_1, \ldots, x_s\}$ with $x_i \neq 0 \pmod p$ such that
\begin{equation}
	\begin{split}
		u_1x_1^2 + \ldots + u_tx_t^2 &= U \pmod p\\
		v_1x_1^3 + \ldots + v_tx_t^3 &= V \pmod p.
	\end{split}
\end{equation}
\end{lemma}

\begin{lemma}\label{explicitWBound} For $p > 3$, $a$, $b$ not both $p$, $|W_i(p,a,b)| \leq 2\sqrt{p}+1.$
\end{lemma}
\begin{proof}

Corollary 2F of~\cite{schmidt} gives
\[
	\left| \sum_{\substack{r=0}}^{p-1} e(\frac{au_ir^2+bv_ir^3}{p}) \right| \leq 2p^{1/2}.
\]
Now
\[
	|W_i(p,a,b)| = \left| \sum_{\substack{r=1}}^{p-1} e(\frac{au_ir^2+bv_ir^3}{p}) \right|
\]
\[
	\leq \left| \sum_{\substack{r=0}}^{p-1} e(\frac{au_ir^2+bv_ir^3}{p}) \right| + 1 \leq 2\sqrt{p} + 1.
\]
\end{proof}

Let $M_t(q)$ be the number of solutions of the system
\begin{equation}\nonumber
	\begin{split}
		 u_1x_1^2 + \ldots + u_tx_t^2 &\equiv 0 \pmod q \\
		 v_1x_1^3 + \ldots + v_tx_t^3 &\equiv 0 \pmod q.
	\end{split}
\end{equation}

\begin{lemma}\label{MpBound} $M_t(p) \geq \frac{1}{p^2}\big((p-1)^t- (p^2-1)(2\sqrt{p}+1)^t\big)$.
\end{lemma}
\begin{proof}
\[
	M_t(p) = \frac{1}{p^2}\sum_{r_1=1}^p \sum_{r_2=1}^p \prod_{i=1}^t W_i(p,r_1,r_2).
\]
We have $W_i(p,p,p)=p-1$ and for $r_1$, $r_2$ not both $p$, $W_i(p,r_1,r_2) \leq 2\sqrt{p}+1$ by Lemma~\ref{explicitWBound}.  Thus
\[
	\left|M_t(p) - \frac{(p-1)^t}{p^2}\right| \leq \frac{1}{p^2} \sum_{r_1=1}^p \sum_{\substack{r_2=1 \\ \{r_1,r_2\} \neq \{p,p\}}}^p \prod_{i=1}^t (2\sqrt{p}+1)
\]
\[
	\leq \frac{1}{p^2}(p^2-1)(2\sqrt{p}+1)^t.
\]
So we have
\[
	M_t(p) \geq \frac{1}{p^2}\big((p-1)^t- (p^2-1)(2\sqrt{p}+1)^t\big).
\]
\end{proof}

Taking $t=7$, we get
\[
	M_7(p) \geq \frac{1}{p^2}\big((p-1)^7- (p^2-1)(2\sqrt{p}+1)^7\big).
\]
This gives that $M_7(p) > 0$ for $p > 40.58$.  This means that we now need only check that Lemma~\ref{localTLemma} holds for each prime smaller than 41.  This is now a finite number of cases to check and thus can be verified by computer.  In the following section, we note several techniques that may be employed to bring the computational difficulty of the task into the realm of feasibility, and in Appendix~\ref{codeAppendix} we provide Sage code for performing the computation.

It is worth noting that $t=7$ appears to only be required for $p=7$.  It seems highly probable that $t=5$ will suffice for all other primes; however, reducing $t$ to 5 weakens the bound of Lemma~\ref{MpBound} to requiring us to check all primes less than 1193, which would require more computation than is feasible, since even after the optimizations of Section~\ref{sec:computationalTechniques}, the algorithm checks $O(p^7)$ distinct forms for solvability to verify Lemma~\ref{MpBound} for all primes up through $p$.

\section{Computational Techniques}\label{sec:computationalTechniques}

First, we note that if every pair $U, V$ modulo $p$ can be represented by the form in $t_0$ variables, then every pair can be represented by $t$ variables for $t>t_0$.  So we will start our search with $t=3$ and store the forms that represent all pairs $(U, V)$ of residue classes mod $p$.  We then need only search higher values of $t$ for the forms that failed to represent all pairs of residue classes with a smaller $t$.

(The methods in this paragraph are closely modeled after those of~\cite{wooleySAEi}.)  By independently substituting $c_ix_i$ for each $x_i$, we can assume each $x_i$ is either 1 or a fixed quadratic nonresidue $c$ modulo $p$.  By rearranging and multiplying by $b^{-1}$ as needed, we can assume that $u_1, \ldots, u_r = 1$, $u_{r+1}, \ldots, u_t = c$ with $r \geq \lceil{t/2}\rceil$.  By multiplying the cubic equation by $v_1^{-1}$ and rearranging, we may assume $1 = v_1 \leq v_2 \leq \ldots \leq v_t$.  By substituting $-x_i$ for $x_i$ as needed, we can assume $1 \leq v_i \leq (p-1)/2$ for each $v_i$ without affecting the $u_i$.

As a final optimization, we note that if the form
\begin{equation}
	\begin{split}
		u_1x_1^2 + \ldots + u_tx_t^2 = U \pmod p\\
		v_1x_1^3 + \ldots + v_tx_t^3 = V \pmod p
	\end{split}
\end{equation}
 represents $p^2-1$ of the possible $p^2$ pairs of residue classes $(U, V)$ modulo $p$, then
\begin{equation}
	\begin{split}
		u_1x_1^2 + \ldots + u_{t+1}x_{t+1}^2 = U \pmod p\\
		v_1x_1^3 + \ldots + v_{t+1}x_{t+1}^3 = V \pmod p
	\end{split}
\end{equation}
will necessarily represent all $p^2$ residue classes, since $(u_{t+1}x_{t+1}^2, v_{t+1}x_{t+1}^3)$ must represent at least two distinct pairs of residue classes, so 
\begin{equation}
	\begin{split}
		u_1x_1^2 + \ldots + u_tx_t^2 = U - u_{t+1}x_{t+1}^2 \pmod p\\
		v_1x_1^3 + \ldots + v_tx_t^3 = V - v_{t+1}x_{t+1}^3 \pmod p
	\end{split}
\end{equation}
 will be solvable for some $(u_{t+1}, v_{t+1})$.  This turns out to be quite useful: a substantial number of forms represent exactly $p^2-1$ pairs of residue classes modulo $p$.

Using these techniques to minimize the computation needed, running the Sage code in Appendix~\ref{codeAppendix} verifies that Lemma~\ref{localTLemma} holds for $p < 41$.  This allows us to conclude the following unconditional form of Theorem~\ref{Spositivity}.

\begin{lemma}\label{unconditionalFrakSPositivity} $\mathfrak{S} > 0.$
\end{lemma}

\section{Conclusion}\label{sec:conclusion}

We have that $R(P) = \mathfrak{S}(Q)J(Q) + O(P^{s-5}(\log P)^{-E})$ by Lemma~\ref{singularDecomp}.  Lemma~\ref{unconditionalFrakSPositivity}, in conjuction with Lemma~\ref{Spositivity}, shows that $\mathfrak{S}(Q)>0$ uniformly over all $u_i$, $v_i$ satisfying the conditions of Theorem~\ref{mainThm} or Theorem~\ref{mainThmv2}. 

The singular integral $J(Q)$ is the same as the one Wooley obtains in the corresponding problem over the integers, so by Lemma 7.4 of \cite{wooleyDiagEq}, there exists a positive constant $C$ such that
\[
	J(Q) = CP^{s-5} + O(P^{s-5}Q^{-1/2}).
\]
In addition, we have the asymptotic upper bound $\mathfrak{S}(Q) \ll 1$ from Lemma~\ref{frakSConverges}.  So we have
\[
	R(P) = CP^{s-5} + O(P^{s-5} (\log P)^{-E})
\]
for $E > 0$, $C > 0$ uniformly.

Thus $R(P)$ is eventually positive.  This can only be true if there is a solution of (\ref{eq:system}) over the primes, so we can conclude Theorems~\ref{mainThm} and~\ref{mainThmv2}.

\section{Acknowledgements}
The author is greatly indebeted to Professor Robert Vaughan for suggesting the problem, for a great deal of guidance and assistance, and for many of the ideas of Sections~\ref{sec:minorArcs} and~\ref{sec:Fbound}.  The author also thanks Trevor Wooley, who suggested the approach taken in Lemma~\ref{theTrick}.

\appendix
\section{Sage Code}\label{codeAppendix}

Code: (SageMath 8.6)

\begin{Verbatim}[fontsize=\small]
for p in prime_range(5,41):
    
    # Find a quadratic non-residue modulo p
    for i in range(1,p):
        if i not in quadratic_residues(p):
            c = i
            break
    
    uv_done = []
    print("p = " + str(p))
    
    for t in range(3,8):
        u = [0] * t
        v = [0] * t
        for number_of_c in range(floor(t/2) + 1): # Set u
            for u_index in range(t):
                if u_index < t - number_of_c:
                    u[u_index] = 1
                else:
                    u[u_index] = c
            skip_v = False
            for v_counter in range(((p-1)/2)^(t-1)): # Set v
                v[0] = 1
                for v_index in range(1,t):
                    v[v_index] = floor(v_counter % ((p-1)/2)^(v_index) / ((p-1)/2)^(v_index-1)) + 1
                    if u[v_index] == u[v_index-1] and v[v_index] < v[v_index-1]:
                        skip_v = True
                if skip_v == True:
                    skip_v = False
                else:
                    # If removing the last coefficients yields a smaller form that
                    # has already passed, add this form to that list and continue
                    if (u[:t-1], v[:t-1]) in uv_done:
                        uv_done.append((deepcopy(u),deepcopy(v)))
                    else:
                        L = []
                        done = False
                        for i in range((p-1)^t):
                            if done:
                                break;
                            x = [None] * t
                            for j in range(t): # Set x
                                x[j] = floor(i % (p-1)^(j+1) / (p-1)^j) + 1
                            a = 0
                            b = 0
                            for k in range(t):
                                a = mod(a + u[k]*x[k]^2, p)
                                b = mod(b + v[k]*x[k]^3, p)
                            inL = False
                            for pair in L:
                                if (pair[0] == a and pair[1] == b):
                                    inL = True
                                    break;
                            # If the pair (a, b) has not already been represented 
                            # by this form, store that it can be
                            if inL == False: 
                                L.append((a,b))
                                if len(L) == p^2:
                                    done = True
                                    
                        # Uncomment this line to print information on each form
                        #print("u: " + str(u) + " v: " + str(v) + " " + str(len(L)))
                        
                        # If the form represents all pairs (a, b), add it to the list
                        if done: 
                            uv_done.append((deepcopy(u), deepcopy(v)))
                        # If the form represents all pairs (a, b) but one, add it
                        elif len(L) == p^2-1 and t < 7:
                            uv_done.append((deepcopy(u), deepcopy(v)))
                        else:
                            if t == 7:
                                print("u: " + str(u) + " v: " + str(v) + "fails.")
print("Search complete")
\end{Verbatim}

Output:
\begin{Verbatim}[fontsize=\small]
p = 5
p = 7
p = 11
p = 13
p = 17
p = 19
p = 23
p = 29
p = 31
p = 37
Search complete
\end{Verbatim}

\end{document}